\documentclass[reqno,a4paper]{article}

\usepackage{a4wide}
\usepackage[centertags]{amsmath}
\usepackage{eucal,dsfont,accents,bbm,amsfonts,amssymb,amsthm,amsopn}

\usepackage[usenames]{color}
\definecolor{citegreen}{rgb}{0,0.6,0}
\definecolor{refred}{rgb}{0.8,0,0}
\usepackage[colorlinks, citecolor=citegreen, linkcolor=refred]{hyperref}

\title{Perelman's Entropy Functional at Type~I Singularities\\of the Ricci Flow}
\author{Carlo Mantegazza and Reto M\"{u}ller}
\date{}

\providecommand{\abs}[1]{\lvert #1\rvert}
\providecommand{\Abs}[1]{\Big\lvert #1\Big\rvert}

\providecommand{\scal}[1]{\langle #1\rangle}

\providecommand{\wt}[1]{\widetilde{#1}}
\providecommand{\wh}[1]{\widehat{#1}}

\newcommand{\Rm}{\mathrm{Rm}}
\newcommand{\Rc}{\mathrm{Rc}}
\newcommand{\RRR}{\mathrm{R}}
\newcommand{\II}{\mathrm{I}}
\newcommand{\HH}{\mathrm{H}}

\newcommand{\RR}{\mathbb{R}}
\newcommand{\NN}{\mathbb{N}}

\newcommand{\sW}{\mathcal{W}}
\newcommand{\sF}{\mathcal{F}}
\newcommand{\sU}{\mathcal{U}}

\newcommand{\eps}{\varepsilon}
\newcommand{\Lap}{\triangle}
\newcommand{\up}{u_{p,T}}
\newcommand{\uph}{\widehat{u}_{p,T}}
\newcommand{\fp}{f_{p,T}}
\newcommand{\fph}{\widehat{f}_{p,T}}
\newcommand{\tp}{\theta_p}
\newcommand{\tfpt}{\theta_{f_{p,T}}}
\newcommand{\Minf}{M_\infty}
\newcommand{\finf}{f_\infty}
\newcommand{\ginf}{g_\infty}
\newcommand{\pinf}{p_\infty}
\newcommand{\dt}{\frac{\partial}{\partial t}}
\newcommand{\ds}{\frac{\partial}{\partial s}}

\def\Xint#1{\mathchoice%
{\XXint\displaystyle\textstyle{#1}}%
{\XXint\textstyle\scriptstyle{#1}}%
{\XXint\scriptstyle\scriptscriptstyle{#1}}%
{\XXint\scriptscriptstyle\scriptscriptstyle{#1}}%
\!\int}%
\def\XXint#1#2#3{{\setbox0=\hbox{$#1{#2#3}{\int}$}%
\vcenter{\hbox{$#2#3$}}\kern-.5\wd0}}%
\def\dashint{\Xint-}%

\theoremstyle{remark}

\theoremstyle{plain}
\newtheorem{lemma}{Lemma}[section]
\newtheorem{prop}[lemma]{Proposition}
\newtheorem{thm}[lemma]{Theorem}
\newtheorem{cor}[lemma]{Corollary}
\newtheorem{defn}[lemma]{Definition}

\newtheorem{ackn}{Acknowledgments\!\!}
\newtheorem{rem}{Remark\!\!}
\newtheorem{rems}{Remarks\!\!}

\numberwithin{equation}{section}

\begin{document}
\maketitle

\begin{abstract}
We study \emph{blow--ups around fixed points} at Type~I singularities of the Ricci flow on closed manifolds using Perelman's $\sW$--functional. First, we give an alternative proof of the result obtained by Naber~\cite{Nab10} and Enders--M\"{u}ller--Topping~\cite{EMT11} that blow--up limits are non--flat gradient shrinking Ricci solitons. Our second and main result relates a \emph{limit $\sW$--density} at a Type~I singular point to the \emph{entropy} of the limit gradient shrinking soliton obtained by blowing--up at this point. In particular, we show that no entropy is lost at infinity during the blow--up process.
\end{abstract}

\section{Introduction and Main Results}

A smooth one--parameter family $(M^n,g(t))$ of complete $n$--dimensional Riemannian manifolds satisfying Hamilton's Ricci flow equation $\dt g(t) = -2\Rc_{g(t)}$ (introduced in~\cite{Ham82}) on a positive time interval $t\in [0,T)$, is said to develop a finite time singularity at time $T<\infty$ if it cannot be smoothly extended past $T$. It is well
known that this is the case if and only if the Riemannian curvature tensor $\Rm$ satisfies $\limsup_{t\to T}\,\sup_M\abs{\Rm(\cdot,t)}_{g(t)}=+\infty$. In the case where all $(M,g(t))$ have bounded curvatures, this is equivalent to
\begin{equation}\label{TypeIlower}
\sup_M\abs{\Rm(\cdot,t)}_{g(t)}\geq\frac{1}{8(T-t)}, \quad \forall t\in[0,T).
\end{equation}
Such a singular Ricci flow on $[0,T)$ is said to be of \emph{Type~I} if there exists a constant $C_{\II}$ satisfying
\begin{equation}\label{TypeIupper}
\sup_M\abs{\Rm(\cdot,t)}_{g(t)}\leq\frac{C_{\II}}{T-t}, \quad \forall t\in[0,T).
\end{equation}
If no such $C_{\II}$ exists, the Ricci flow is said to be of \emph{Type~II}. A natural line to study finite time singularities is to take blow--ups based at a fixed (singular) point $p\in M$. This can be done in two slightly different ways.

\begin{defn}[Dynamical blow--up and blow--up sequences]\label{defseq}
Let $(M,g(t))$ be a solution to the Ricci flow on a finite time interval $[0,T)$.
\begin{enumerate}
\item The pointed flow $(M,\wt{g}(s),p)$ with
\begin{equation}
\wt{g}(s):=\frac{g(t)}{T-t}, \quad s(t):=-\log(T-t)\in [{-\log T},+\infty)
\end{equation}
is called \emph{dynamical blow--up} based at $p\in M$. It satisfies the evolution equation
\begin{equation*}
\ds \wt{g}(s) = (T-t)\dt\Big(\frac{g(t)}{T-t}\Big)=-2\Rc_{g(t)}+\frac{g(t)}{T-t} =-2\Rc_{\wt{g}(s)}+\wt{g}(s).
\end{equation*}
\item For $\lambda_j\to\infty$, we define a \emph{blow--up sequence} based at $p\in M$ to be the family of pointed rescaled Ricci flows $(M,g_j(s),p)$ defined by
\begin{equation}\label{sequential}
g_j(s):=\lambda_j g(T+\tfrac{s}{\lambda_j}), \quad s\in[-\lambda_jT,0).
\end{equation} 
\end{enumerate}
\end{defn}

In~\cite{Per02}, Perelman introduced two important monotone quantities that can be used to study the singularity formation. On the one hand, he defined a \emph{reduced volume quantity} whose monotonicity under the Ricci flow might be seen as a parabolic analog of the classical Bishop--Gromov volume ratio monotonicity. On the other hand, he introduced the \emph{$\sW$--entropy functional} whose monotonicity might be interpreted as an intrinsic analog of Huisken's monotonicity formula for the mean curvature flow~\cite{Hui90}. Both quantities have their advantages and disadvantages compared to the other one. While the reduced volume seems to behave better under topological surgery and is therefore used as the main tool in the study of Ricci flow with surgery on three--manifolds (see e.g.~\cite{Per03,MT07}), the $\sW$--entropy is more useful for regularity, stability and uniqueness questions (see in particular~\cite{HN12, SW10, Ache12}). The latter is mainly due to the fact that $\sW$ satisfies a {\L}ojasievicz--Simon inequality, a result that relies on the property that the functional is locally analytic.

In the case of a Type~I Ricci flow, Naber~\cite{Nab10} and Enders--M\"{u}ller--Topping~\cite{EMT11} proved smooth convergence of a blow--up sequence to a non--trivial gradient shrinking soliton, using a version of Perelman's reduced volume \emph{based at the singular time} $T$. One of the goals of this article is to reprove their result by means of Perelman's $\sW$--entropy functional, given by
\begin{equation}
\sW(g,f,\tau):= \int_M \Big(\tau(\RRR_g+\abs{\nabla f}_g^2)+f-n\Big)\frac{e^{-f}}{(4\pi\tau)^{n/2}}dV_g,
\end{equation}
under the constraint 
\begin{equation}\label{constraint}
\int_M \frac{e^{-f}}{(4\pi\tau)^{n/2}}dV_g=1.
\end{equation}

The main reasons for finding this alternative proof are the following. First, our new approach gives an extra quantitative information on the $\sW$--entropy of the limit gradient shrinking Ricci soliton. Moreover, our proof is a first step towards obtaining uniqueness of \emph{noncompact} blow--up limits, which is still an open problem. Finally, our line of analysis seems to be better suited to deal with the Type~II singularity case, for which a reduced volume based at the singular time is not known to exist.

Let us now describe our method and results in more detail. The most important step is the choice of a suitable test function for the $\sW$--functional. As proposed by Ilmanen, we let $\tau(t):=T-t$ be the remaining time to the finite time singularity and we choose $f(\cdot, t)=\fp(\cdot, t)$ in such a way that $\up(\cdot,t):=\frac{e^{-\fp(\cdot,t)}}{(4\pi\tau)^{n/2}}$ is an {\em adjoint heat kernel based at the singular time} $(p,T)$, that is, a locally smooth limit of solutions of the {\em backward} parabolic PDE
\begin{equation*}
\frac{\partial}{\partial t} u=-\Lap u+\RRR_g u,
\end{equation*}
all converging (as distributions) to a Dirac $\delta$--measure at $p\in M$ at times closer and closer to the singular time $T$ (see Section~\ref{secfpT} for a precise definition and an existence proof). Then, the monotonicity of the $\sW$--entropy functional for solutions of the above backward parabolic PDE (discovered by Perelman, see~\cite{Per02}) allows us to get information on the singular behavior of the flow around the point $p$ (notice that any $\fp$ constructed in this manner satisfies the above normalization constraint \eqref{constraint}).

The main reason for working with these adjoint heat kernels is that they concentrate around the chosen base point and therefore allow us to ``zoom into the singular region''. A different approach, studied by Le and Sesum~\cite{LS10} (assuming also uniformly
bounded scalar curvature), is to work with the minimizers $f_{\min}(t)\in C^\infty(M)$ of the ``frozen'' $\sW$--entropy functional $f\mapsto\sW(g(t),f,T-t)$ under the constraint \eqref{constraint}. However, in \cite{LS10} they consider a pointed limit with basepoints defined to be the minimum points of $f_{\min}(t)$, while for our purpose it is important to be able to fix the singular point $p\in M$ to draw conclusions about the geometry near the singularity (see Section~\ref{secfpT} for more details on this point).

For technical reasons and to avoid the complications due to the 
possible non--uniqueness of the function $\fp$, we actually consider and work with 
the functions $\fph(t):M\to\RR$ obtained by minimization at any fixed $t\in[0,T)$ of the 
$\sW$--entropy among the family of all limits $\fp$.

\begin{defn}\label{minimizer}
Let $\sF_p$ be the family of functions such that $\up(\cdot,t)=\frac{e^{-\fp(\cdot,t)}}{[4\pi(T-t)]^{n/2}}$ is an adjoint heat kernel based at the singular time as above (or as in Definition~\ref{defahk} to be more precise). Since $\sF_p$ is compact (according to
Lemma~\ref{compactlemma}), for every $t\in[0,T)$ there exists a minimizer $\fp^t$ among all $\fp\in\sF_p$ of the $\sW$--entropy functional $\fp\mapsto\sW(g(t),\fp(\cdot,t),T-t)$, at a fixed time $t$. Then, for each $t\in[0,T)$, we define the smooth functions $\fph(t):M\to\RR$ by $\fph(\cdot,t)=\fp^t(\cdot,t)$.
\end{defn}

From Perelman's entropy formula and Li--Yau--Harnack type
inequality~\cite{Per02}, we obtain the following monotonicity and
nonpositivity result.

\begin{prop}[Monotonicity of $\sW(g(t),\fph(t),\tau(t))$]\label{monoprop}
Letting $\fph(t):M\to\RR$ be as in Definition~\ref{minimizer} and $\tau(t):=T-t$, we define
$\tp(t):=\sW(g(t),\fph(t),\tau(t))$. Then, the function $\theta_p:[0,T)\to\RR$ is nonpositive and non--decreasing along the Ricci flow with derivative
\begin{equation}\label{monotonicity}
\dt \tp(t) = 2\tau\int_M \Abs{\Rc_{g(t)} + \nabla^2 \fph(t) -\frac{g(t)}{2\tau}}_{g(t)}^2\frac{e^{-\fph(t)}}{(4\pi\tau)^{n/2}}dV_{g(t)} \geq 0
\end{equation}
at almost every $t\in[0,T)$.
\end{prop}
This proposition will be proved in Section~\ref{secW}. In particular,
it implies that the limit $\Theta(p):=\lim_{t\to T}\theta_p(t)$ exists
and is nonpositive, we will  call it {\em limit $\sW$--density} at $p\in M$.
From~\eqref{monotonicity}, by a simple
rescaling argument, we see that the dynamical blow--up $\wt{g}(s)$
defined in Point~1 of Definition~\ref{defseq} satisfies
\begin{equation}\label{L2eqdyn}
\lim_{j\to\infty}\int_j^{j+1}\int_M\Abs{\Rc_{\wt{g}(s)}+\nabla^2\wt{f}(s)-\frac{\wt{g}(s)}{2}}^2_{\wt{g}(s)}e^{-\wt{f}(s)}dV_{\wt{g}(s)}\, ds = 0,
\end{equation}
where $\wt{f}(s)=\fph(t(s))$ and $s(t)=-\log(T-t)$ as before. Similarly, taking a compact time interval $[S_0,S_1]\subset(-\infty,0)$ and setting $f_j(s):=\fph(t(s))$ for $s(t)=\lambda_j(t-T)$, we obtain for a blow--up sequence $(M,g_j(s))$ as in formula~\eqref{sequential}
\begin{equation}\label{L2eqseq}
\lim_{j\to\infty} \int_{S_0}^{S_1} \int_M \Abs{\Rc_{g_j(s)} +
  \nabla^2f_j(s) -\frac{g_j(s)}{2\abs{s}}}^2_{g_j(s)}e^{-f_j(s)}dV_{g_j(s)}\, ds
= 0.
\end{equation}
More details about how to derive equations~\eqref{L2eqdyn} and~\eqref{L2eqseq} are provided in Section~\ref{secResc}.

Based on the above, it is natural to expect that a suitable
subsequence of $(M,\wt{g}(s_j),\wt{f}(s_j),p)$ with $s_j\to\infty$, or
of $(M,g_j(-1),f_j(-1),p)$, respectively, converges (possibly in a
weak sense) to a gradient shrinking Ricci soliton
$(\Minf,\ginf,\finf,\pinf)$ for every fixed $p\in M$, that is, a complete Riemannian manifold $(\Minf,\ginf)$ satisfying
\begin{equation}\label{solitoneq}
\Rc_{\ginf}+\nabla^2\finf=\frac{\ginf}{2}.
\end{equation}

In this paper, we prove this for the case of Type~I Ricci flows. The case of general singularities (in low dimensions) will be studied elsewhere. In particular, we will see in Section~\ref{secfpT} that for Type~I Ricci flows we have the bounds
\begin{equation}\label{mainbounds}
\wh{C}e^{-d^2_{g(t)}(p,q)/\wh{C}(T-t)}\leq e^{-\fp(q,t)}\leq \bar{C}e^{-d^2_{g(t)}(p,q)/\bar{C}(T-t)}
\end{equation}
for every $\fp\in\sF_p$ and $(q,t)\in M\times [0,T)$, where $\wh{C}$, $\bar{C}$
are positive constants depending only on $n$ and $C_{\II}$ (see
Proposition~\ref{effectiveprop} and Proposition~\ref{prop.Gaussian}). Once we have the positive bound from below, by uniform estimates in $C^\infty_{loc}(M\times[0,T))$ for
the family of functions $\fph(t)$, we can pass to the limit in equations~\eqref{L2eqdyn} and~\eqref{L2eqseq}, obtaining a complete and smooth limit Riemannian manifold
$(\Minf,\ginf)$ and a nonzero limit term $e^{-\finf}$ which then 
implies that $(\Minf,\ginf)$ satisfies the gradient shrinking soliton equation~\eqref{solitoneq}.

This is stated more precisely in the following theorem, previously obtained by
Naber~\cite{Nab10} and Enders--M\"{u}ller--Topping~\cite{EMT11} using
different ideas (as explained above).

\begin{thm}[Blow--ups at Type~I singularities are shrinkers, cf.~\cite{EMT11,Nab10}]\label{mainthm}
Let $(M,g(t))$ be a compact singular Type~I Ricci flow, that is, a
compact Ricci flow satisfying both inequalities~\eqref{TypeIlower}
and~\eqref{TypeIupper}. Let $p\in M$ and
$\fph(t)$ be as in Definition~\ref{minimizer}. 

\begin{enumerate}
\item Let $(M,\wt{g}(s),p)$ be the dynamical blow--up as in Point~1 of Definition~\ref{defseq} and $\wt{f}(s)=\fph(t)$ for $s=-\log(T-t)$. For every family of disjoint intervals $(a_k,b_k)$ with $\sum_{k\in\NN} (b_k-a_k)=+\infty$ there are $s_j\in\cup_{k\in\NN}(a_k,b_k)$ with $s_j\to\infty$ such that $(M,\wt{g}(s_j),\wt{f}(s_j),p)$ converges smoothly in the pointed Cheeger--Gromov sense to a {\em normalized} gradient shrinking Ricci soliton $(\Minf,\ginf,\finf,\pinf)$, that is, a complete Riemannian manifold $(\Minf,\ginf)$ satisfying~\eqref{solitoneq}, where $\finf:\Minf\to\RR$ is a smooth function with $\int_{\Minf} \frac{e^{-\finf}}{(4\pi)^{n/2}}dV_{\ginf}=1$.

\item Let $\lambda_j\to\infty$, let $(M,g_j(s),p)$ be the corresponding blow--up sequence as in Point~2 of Definition~\ref{defseq}, and set $f_j(s)=\fp(t)$ for $s=\lambda_j(t-T)$. Then there is a subsequence (not relabeled) of indices $j\in\NN$ such that the $(s=-1)$--slices $(M,g_j(-1),f_j(-1),p)$ of the blow--up sequence converge smoothly in the pointed Cheeger--Gromov sense to a normalized gradient shrinking Ricci soliton $(\Minf,\ginf,\finf,\pinf)$, as above.
\end{enumerate}
\end{thm}

To remind the reader of the definition of the Cheeger--Gromov convergence, the precise statement is that there exist an exhaustion of $\Minf$ by open sets $U_j$ containing $\pinf$ and smooth embeddings $\phi_j:U_j\to M$ with $\phi_j(\pinf)=p$
such that $(\phi_j^*\wt{g}(s_j),\phi_j^*\wt{f}(s_j))$ or $(\phi_j^*g_j(s),\phi_j^*f_j(s))$, respectively, converge smoothly to $(\ginf,\finf)$ on every compact subset of $\Minf$.

Combining this result with the bounds for $e^{-\fph(t)}$ in~\eqref{mainbounds}, we get a new proof of the following nontriviality statement which was previously obtained by Enders--M\"{u}ller--Topping~\cite{EMT11}, using Perelman's pseudolocality theorem (see~\cite{Per02}).

\begin{thm}[Nontriviality of blow--up limits around singular points, cf.~\cite{EMT11}]\label{nontriv} 
If in addition to the assumptions of Theorem~\ref{mainthm} the point
$p\in M$ is a singular point, that is, there does not exist any neighborhood
$U_p\ni p$ on which $\abs{\Rm(\cdot,t)}_{g(t)}$ stays bounded as $t\to
T$, then the limits obtained in Theorem~\ref{mainthm} are non--flat.
\end{thm}

Along the way of proving the Theorems~\ref{mainthm} and~\ref{nontriv}, we will also obtain the following new and main result. We state (and prove) this only for the case of dynamical blow--ups, but all the arguments can be easily adopted to the case of sequential blow--ups.

\begin{thm}[No loss of entropy]\label{noloss}
For any sequence of pointed rescaled manifolds $(M,\wt{g}(s_j),p_j)$
and functions $\wt{f}(s_j)=\widehat{f}_{p_j,T}(t(s_j))$
converging locally smoothly to some gradient shrinking Ricci soliton
$(\Minf,\ginf,\pinf)$ and relative potential function $\finf:\Minf\to\RR$, we have
\begin{align*}
\sW(\ginf,\finf)&:=\int_{\Minf} \big(\RRR_{\ginf}
+\abs{\nabla \finf}^2_{\ginf}+\finf-n\big)
\frac{e^{-\finf}}{(4\pi)^{n/2}}\,dV_{\ginf}\\
&=\lim_{j\to\infty}\int_{M}\big(\RRR_{\wt{g}(s_j)}
+\abs{\nabla\wt{f}(s_j)}^2_{\wt{g}(s_j)}
+\wt{f}(s_j)-n\big)\,\frac{e^{-\wt{f}(s_j)}}{(4\pi)^{n/2}}\,dV_{\wt{g}(s_j)}.
\end{align*}
\end{thm}

Note that $\sW(\ginf,\finf)$ is well--defined, as explained in Section~\ref{secShrinker} below.

\begin{rems}\ 
\begin{enumerate}
\item For simplicity, we have assumed in the three theorems above that our Type~I Ricci flow $(M,g(t))$ is compact. However, up to technical modifications, all our results should go though in the complete, non--compact case assuming that the initial manifold is uniformly non--collapsed. In particular, the bounds for $f_{p,T}$ in \eqref{mainbounds} ensure that the $\sW$--density is finite in the Type~I case also on non--compact manifolds, all integrals exist, and all partial integrations are justified. 

\item\label{rem1} Actually, a much weaker upper bound for $\fp$ than the one in~\eqref{mainbounds} (namely the bound given in Corollary~\ref{rinf}) is sufficient to prove the Theorems~\ref{mainthm}--\ref{noloss}. We will always give proofs using only this weak bound, as there is a chance that one might obtain it also in the Type~II case, where the Gaussian upper bound in~\eqref{mainbounds} seems too strong to hope for.

\item Instead of working with the functions $\fph(t)$, we can also work with some
fixed $\fp\in\sF_p$. Different choices of $\fp$ or of different
sequences $s_j\to\infty$ might
a--priori lead to different gradient shrinking Ricci solitons in the
limit. However, we will see that all the limits based at $p\in M$ must
have the same $\sW$--entropy $\sW(g_\infty,f_\infty)=\Theta(p)$ (see
the comments after Corollary~\ref{thetaprop}).

\item Actually, a slightly stronger version of the conclusion in Theorem~\ref{mainthm} holds true: we do not only have convergence of time--slices, but rather convergence on compact sets in \emph{space--time} in the sense of Hamilton's compactness theorem for Ricci flows~\cite{Ham95}. We leave the details to the reader.

\item This line of analysis of Type~I Ricci flow singularities was suggested in~\cite{CHI} and it is the analogue of the standard way to deal with Type~I singularities of mean curvature flow of hypersurfaces in $\RR^{n+1}$ (at least in the positive mean curvature case), mainly developed by Huisken (see~\cite{Hui90} and also Stone~\cite{Sto94}). In this context, Huisken's monotonicity formula
\begin{equation*}
\frac{d}{dt}\int_{M} \frac{e^{-\frac{\abs{x-x_0}^2}{4(T-t)}}}{[4\pi(T-t)]^{n/2}}\,dV_M = -\int_{M} \Abs{\HH+\frac{\scal{x-x_0\mid\nu}}{2(T-t)}}^2\,\frac{e^{-\frac{\abs{x-x_0}^2}{4(T-t)}}}{[4\pi(T-t)]^{n/2}}\,dV_M
\end{equation*}
plays the role of Perelman's entropy formula~\eqref{permon}. 
In the mean curvature flow case the ``test functions'' are the solutions of the backward heat equation in $\RR^{n+1}$ (in the ambient space) which can be written explicitly --- in contrast to the solutions of the adjoint heat equation on the evolving manifold $(M,g(t))$ used here --- for instance, $u(x,t)=[4\pi(T-t)]^{-n/2}e^{-\frac{\abs{x-x_0}^2}{4(T-t)}}$. Moreover, the actual presence of an ambient space avoids the necessity of the limit procedure that we mentioned above to get a test function converging to a Dirac $\delta$--measure at the singular time.
\end{enumerate}
\end{rems}

The paper is organized as follows. In the next section, we construct the adjoint heat kernels $\up$ based at the singular time as locally smooth limits of adjoint heat kernels $u_{p,s_i}$ based at $s_i\to T$ (see Lemma~\ref{roughlemma}). In the special case of Type~I Ricci flows, we then derive effective bounds from below for $u_{p,s_i}$ and $\up$ in Proposition~\ref{effectiveprop}, suitable to take a limit of~\eqref{L2eqdyn},~\eqref{L2eqseq} in the blow--up procedure. Moreover, we also obtain Gaussian upper bounds for $u_{p,s_i}$ and $\up$ (Proposition~\ref{prop.Gaussian}) using a recent result of Hein and Naber~\cite{HN12}. This Gaussian upper bounds immediately imply the weaker bound for $\fp$ from Corollary~\ref{rinf}, which is the estimate we actually work with (see Remark~\ref{rem1} above). In particular, this weaker estimate is still strong enough to ensure that no $\sW$--entropy is ``lost'' in the blow--up process.

In Section~\ref{secW}, we study the $\sW$--density functions 
$\theta_p(t)=\sW(g(t),\fph(t),\tau(t))$ and prove Proposition~\ref{monoprop} as well
as some related results. Then, we briefly collect some facts about
gradient shrinking Ricci solitons and the properties of their
$\sW$--entropy in Section~\ref{secShrinker}. Finally, combining the results and estimates from Sections~\ref{secfpT}--\ref{secShrinker}, we prove the Theorems~\ref{mainthm},~\ref{nontriv} and~\ref{noloss} in Section~\ref{secResc}.

\begin{ackn} We thank Robert Haslhofer, Hans--Joachim Hein and Miles Simon for fruitful and interesting discussions. Both authors were partially supported by the Italian FIRB Ideas ``Analysis and Beyond''. In addition, RM was partially financed by an Imperial College Junior Research Fellowship. 
\end{ackn}

\section{The Functions $\fp(t)$ --- Existence and Estimates}\label{secfpT}

A natural first approach (see Le and Sesum~\cite{LS10}) would be to work with the minimizers $f_{\min}(t):M\to\RR$ of the "frozen" functional $f\mapsto\sW(g(t),f,T-t)$ over all $f$ satisfying the constraint $\int_M e^{-f}dV=[4\pi(T-t)]^{n/2}$. The minimum is actually attained on closed manifolds by some smooth function and it is known that
Perelman's entropy function $\mu(g(t)) :=\sW(g(t),f_{\min}(t),T-t)$ is nonpositive and
non--decreasing in $t\in[0,T)$ along the Ricci flow. Moreover, the
following estimate holds
\begin{equation*}
\dt \mu(g(t)) \geq 2\tau\int_M \Abs{\Rc_{g(t)} + \nabla^2f_{\min}(t)
-\frac{g(t)}{2\tau}}^2_{g(t)}\frac{e^{-f_{\min}(t)}}{(4\pi\tau)^{n/2}}\,dV_{g(t)}\geq 0,
\end{equation*}
where $\tau=T-t$, see~\cite{Per02,LS10}.

Rescaling the flow $g(t)$, we obtain formulas similar to~\eqref{L2eqdyn}
and~\eqref{L2eqseq}, for instance, by dynamically rescaling the flow as in Definition~\ref{defseq}, we get
\begin{equation*}
\lim_{j\to\infty}\int_j^{j+1}\int_M\Abs{\Rc_{\wt{g}(s)}+\nabla^2 f_{\min}(s)-\frac{\wt{g}(s)}{2}}^2_{\wt{g}(s)} e^{-f_{\min}(s)}dV_{\wt{g}(s)}\,ds = 0,
\end{equation*}
hence, naively, the family of minimizers $f_{\min}(t)$ seems like a good choice for 
singularity analysis. Unfortunately, we do not know whether the minimizers are uniformly bounded above (or equivalently, $e^{-f_{\min}(t)}$ is bounded below) in a
neighborhood of the singular point $p\in M$ which we use as the center of the blow--up. Thus, we could lose the information contained in the integral when taking a limit (and hence fail to conclude the soliton equation~\eqref{solitoneq} for the limit manifold) even in the Type~I case, as the factor $e^{-f_{\min}(s)}$ might go to zero.

It is indeed an interesting question whether or not such a bound
actually exists, that is, whether or not the minimizer $f_{\min}(t)$
asymptotically "sees" any singularity. We actually believe that this is not always the case in general.

To overcome this difficulty, we work with a different choice of ``test 
functions'' which are concentrated around $p\in M$. Namely, we define $f(t)=\fp(t):M\to\RR$ via the requirement that $\up(\cdot,t)=\frac{e^{-\fp(\cdot,t)}}{(4\pi\tau)^{n/2}}$ is a limit of adjoint heat kernels converging to a Dirac $\delta$--measure at $p\in M$, at times closer and closer to the singular time $T$.

In this section, we first prove that the set of such limits is always nonempty and then 
we derive effective estimates in the case of compact Type~I Ricci flows. The nonpositivity and monotonicity results (stated in Proposition~\ref{monoprop}) are proved below in Section~\ref{secW}.

\begin{defn}[Adjoint heat kernels]
Let $(M,g(t))$ be a Ricci flow defined for $t\in[0,T)$. For any $p\in
M$ and $s\in[0,T)$, we denote by $u_{p,s}:M\times[0,s)\to\RR$ the unique
smooth positive solution of the adjoint heat equation
\begin{equation*}
\dt u=-\Lap u+\RRR_g u
\end{equation*}
satisfying $\lim_{t\to s}u_{p,s}(\cdot,t)=\delta_p$ as measures
(here $\delta_p$ is the Dirac $\delta$--measure on $M$ based at
$p$). We call such a function the \emph{adjoint heat kernel based at} $(p,s)$. 
As the functions $u_{p,s}$ are positive, we can 
define $f_{p,s}:M\times[0,s)\to\RR$ by
\begin{equation*}
f_{p,s}(q,t):=-\log(u_{p,s}(q,t))-\frac{n}{2}\log[4\pi(s-t)],
\end{equation*}
that is, $u_{p,s}(q,t)=\frac{e^{-f_{p,s}(q,t)}}{[4\pi(s-t)]^{n/2}}$.
\end{defn}
For an existence and uniqueness proof of the solutions $u_{p,s}$, 
see~\cite[Chapter~24]{RFTA3}. Moreover, the kernels $u_{p,s}$ satisfy
the semigroup property, the equality $\int_M u_{p,s}(\cdot,t)\,dV_{g(t)}=1$ for every $t\in[0,s)$ and a Harnack estimate. A consequence of the latter is the smooth dependence on the point $p\in M$, once restricting $u_{p,s}$ to any compact subset of $M\times[0,s)$ (see again~\cite[Chapter~24]{RFTA3}).

Our first step is to obtain a smooth limit of a sequence of functions $u_{p,s}$
as $s\to T$, using rough standard interior parabolic estimates.

\begin{lemma}[Limits of adjoint heat kernels]\label{roughlemma}
For every sequence of adjoint heat kernels $u_{p,s_i}$ based at
$(p,s_i)$ with $s_i\to T$, there exists a convergent subsequence,
smoothly converging on every compact subset of $M\times[0,T)$ to a
positive and smooth limit function $\up$. Moreover, any $\up$ solves
the adjoint heat equation $\dt u=-\Lap u+\RRR_g u$ on $M\times[0,T)$ and
satisfies $\int_M \up(q,t)\,d V_{g(t)}=1$ for every $t\in[0,T)$.
\end{lemma}

\begin{proof}
Let $u_{p,s_i}$ be a sequence as above. For every compact $K\subset M\times [0,T)$, 
we let $\bar{t}=\sup_{(q,t)\in K} t$ and, for $i\in\NN$ large enough, 
we consider the positive functions $\widetilde{u}_{p,s_i}:M\to\RR$ defined by
$\widetilde{u}_{p,s_i}(q)=u_{p,s_i}(q,(T+\bar{t})/2)$.

By the semigroup property of the kernels $u_{p,s}$, we have
$\int_M\widetilde{u}_{p,s_i}\,dV_{g{(T+\bar{t})/2}}=1$ and
\begin{equation*}
u_{p,s_i}(q,t)=\int_M
u_{r,(T+\bar{t})/2}(q,t)\widetilde{u}_{p,s_i}(r)\,dV_{g{(T+\bar{t})/2}}(r).
\end{equation*}
Hence, by the smooth dependence of the kernels on their basepoint, the
compactness of $M$, and the fact that $\int_M \widetilde{u}_{p,s_i}\,dV_{g(T+\bar{t})/2}=1$, we can uniformly estimate all the derivatives of $u_{p,s_i}\vert_K$. By the Arzel\`a--Ascoli theorem, we can then extract a
converging subsequence in $C^\infty(K)$ and, using a diagonal argument, this can clearly be
done for a family of compact subsets exhausting $M\times[0,T)$, yielding a smooth limit function $\up$.

The facts that any limit function $\up$ is nonnegative, 
satisfies the adjoint heat equation $\dt u=-\Lap u+\RRR_g u$ on $M\times[0,T)$, and that 
$\int_M \up(q,t)\,d V_{g(t)}=1$ for every $t\in[0,T)$, are
consequences of the locally smooth convergence.

Finally, the strong maximum principle implies that every function $\up$ cannot vanish at any point, hence they are all strictly positive.
\end{proof}

\begin{rem}
Notice that in this argument, the estimates are independent of the
point $p\in M$, hence the compactness conclusion holds also for any
sequence of functions $u_{p_i,s_i}$ with $s_i\to T$.
\end{rem}

We can now make Definition~\ref{minimizer} more precise.

\begin{defn}[Adjoint heat kernels based at the singular time]\label{defahk}
We call a smooth limit function $\up$ as in Lemma~\ref{roughlemma}, an \emph{adjoint heat kernel based at the singular time $(p,T)$} and we define $\sU_p$ to be the family of all such possible limits. Moreover, we also define the family $\sF_p$ of smooth functions
$\fp:M\times[0,T)\to\RR$ such that $\up=\frac{e^{-\fp}}{[4\pi(T-t)]^{n/2}}$ belongs to $\sU_p$.
\end{defn}

We notice that by the lemma above and the subsequent remark, we obtain the following compactness property.

\begin{lemma}[Compactness of $\sU_p$]\label{compactlemma}
The family of smooth functions $\sU_p$ and the whole union $\sU=\cup_{p\in M}\sU_p$ 
are compact in the topology of $C^\infty$--convergence on
compact subsets of $M\times[0,T)$.
\end{lemma}

A consequence of this lemma is that there exist a ``minimal'' function $\underline{u}:M\times[0,T)\to\RR$ and a ``maximal'' function $\bar{u}:M\times[0,T)\to\RR$ such that $\underline{u}(q,t)=\inf_{u_{p,T}\in \sU}
u_{p,T}(q,t)$ and $\bar{u}(q,t)=\sup_{u_{p,T}\in \sU}u_{p,T}(q,t)$. By the
above compactness properties, it turns out that $\underline{u}$ must
be positive, by strong maximum principle, and 
$\bar{u}:M\times[0,T)\to\RR$ must be locally bounded above.

Then, by this uniform local positive bound from below, it is straightforward to
conclude that the families of smooth functions $\sF_p$ share the same
compactness property.

\begin{cor}[Compactness of $\sF_p$]\label{fboundscor} 
The family of smooth functions $\sF_p$ and the whole union $\sF=\cup_{p\in M}\sF_p$ are compact in the topology of $C^\infty$--convergence on
compact subsets of $M\times[0,T)$.
\end{cor}

Another consequence is that the family $\sF_p$ can be characterized as
the union of all the possible smooth 
limits $\fp$ of functions $f_{p,s}$ as $s\to
T$ (in the sense of $C^\infty$--convergence on compact subsets of
$M\times[0,T)$).

As the functions $f_{p,s}$ satisfy the evolution equation
\begin{equation*}
\dt f_{p,s}=-\Lap f_{p,s}+\vert\nabla f_{p,s}\vert^2_{g}-\RRR_g +\frac{n}{2(s-t)},
\end{equation*}
any function $\fp\in\sF_p$ satisfies
\begin{equation}\label{evolutionfp}
\dt \fp=-\Lap \fp+\vert\nabla \fp\vert^2_{g}-\RRR_g +\frac{n}{2(T-t)}.
\end{equation}
Then, the following result of Perelman~\cite[Corollary~9.3]{Per02} 
(see also Ni~\cite{Ni06} for a detailed proof) holds.

\begin{prop}[Perelman~\cite{Per02}]\label{nilemma2}
For any function $f_{p,s}$ such that 
$u_{p,s}(q,t)=\frac{e^{-f_{p,s}(q,t)}}{[4\pi(s-t)]^{n/2}}$ is a positive solution of 
the adjoint heat equation converging to a $\delta$--measure as $t\to s$, we have
\begin{equation*}
(s-t)\big(2\Lap f_{p,s}(q,t) -\abs{\nabla f_{p,s}(q,t)}^2_{g(t)} +\RRR_{g(t)}(q,t)\big)+f_{p,s}(q,t)-n\leq 0,
\end{equation*}
for every $q\in M$ and $t\in[0,s)$. Hence, passing to the limit as $s\to T$,
we obtain for any function $\fp\in\sF_p$ and every $(q,t)\in
M\times[0,T)$ that
\begin{equation*}
(T-t)\big(2\Lap \fp(q,t)
-\abs{\nabla \fp(q,t)}^2_{g(t)}
+\RRR_{g(t)}(q,t)\big)+\fp(q,t)-n\leq 0.
\end{equation*}
\end{prop}

As a consequence, one gets the following effective lower bounds for $\up\in\sU_p$,
which for simplicity we only prove in the Type~I situation where the proof is very easy (see for example Cao and Zhang~\cite{CZ11} for a more general result).

\begin{prop}[Gaussian lower bounds for $\up$ in the Type~I case]\label{effectiveprop}
Let $(M,g(t))$ be a compact Ricci flow on $[0,T)$ satisfying the Type~I condition~\eqref{TypeIupper}. Then there exists a positive constant $\wh{C}$
depending only on the dimension $n$ and on the Type~I constant
$C_{\II}$, defined by inequality~\eqref{TypeIupper}, such that 
\begin{equation*}
u_{p,s}(q,t)\geq \frac{\wh{C}}{[4\pi(s-t)]^{n/2}}\;e^{-d^2_{g(t)}(p,q)/\wh{C}(s-t)},
\end{equation*}
for every adjoint heat kernel $u_{p,s}$ and every point $(q,t)\in M\times[0,s)$. Hence, writing $\tau=T-t$,
\begin{equation}\label{Gausslower}
\up(q,t)\geq \frac{\wh{C}}{(4\pi\tau)^{n/2}}\;e^{-d^2_{g(t)}(p,q)/\wh{C}\tau},
\end{equation}
for every $\up\in\sU_p$ and $(q,t)\in M\times[0,T)$.
\end{prop}

\begin{proof}
We use Perelman's differential Harnack inequality~\cite{Per02} in the integrated version (see e.g. Corollary 3.16 in~\cite{Mul06}). This yields $f_{p,s}(q,t)\leq \ell_{p,s}(q,t)$, where $\ell_{p,s}$ is Perelman's backwards reduced length given by
\begin{equation*}
\ell_{p,s}(q,t):=\inf_{\gamma}\bigg\{\frac{1}{2\sqrt{s-t}}\int_t^s\sqrt{s-\sigma}\Big(\abs{\tfrac{\partial}{\partial\sigma}\gamma(\sigma)}^2_{g(\sigma)}+\RRR_{g(\sigma)}(\gamma(\sigma))\Big)d\sigma\bigg\},
\end{equation*}
where the infimum is taken over all curves $\gamma:[t,s]\to M$ with
$\gamma(t)=q$ and $\gamma(s)=p$. To estimate $\ell_{p,s}$ from above,
let $\wh{\gamma}(\sigma)$ be a $g(t)$--geodesic from $q$ at time $t$
to $p$ at time $\frac{t+s}{2}$ and $\wh{\gamma}(\sigma)\equiv p$ for
$\sigma\in[\frac{t+s}{2},s]$. Clearly, $\wh{\gamma}(\sigma)$ is a
candidate for the infimum. Using the Type~I condition, which implies
$\abs{\Rc}_{g(t)}\leq \frac{nC_{\II}}{s-t}$ on $[t,\frac{s+t}{2}]$ and
therefore
$\bigl\vert{\tfrac{\partial}{\partial\sigma}\wh{\gamma}(\sigma)}\bigr\vert^2_{g(\sigma)}
\leq e^{2nC_{\II}}\bigl\vert{\tfrac{\partial}{\partial\sigma}\wh{\gamma}(\sigma)}\bigr\vert^2_{g(t)}$,
we estimate
\begin{equation}\label{ellest}
f_{p,s}(q,t)\leq \ell_{p,s}(q,t)\leq C_1\frac{d^2_{g(t)}(p,q)}{s-t}+C_2,
\end{equation}
where $C_1$ and $C_2$ are two positive constants depending only on
$C_{\II}$ and $n$. The claim follows by substituting $f_{p,s}$ in the expression for
$u_{p,s}$ and estimating.
\end{proof}

Next, we prove similar uniform upper bounds on the adjoint heat kernels $u_{p,s}$ 
under the Type~I assumption. This is a consequence of the recent work of Hein and Naber~\cite{HN12}.

\begin{prop}[Gaussian upper bounds for $\up$ in the Type~I case]\label{prop.Gaussian}
Let $(M,g(t))$ be a compact Ricci flow on $[0,T)$ satisfying the Type~I condition~\eqref{TypeIupper}. Then there exists a positive constant $\bar{C}$, depending only on $n$, the Type~I constant $C_{\II}$ and the initial manifold $(M,g(0))$, such
that for any adjoint heat kernel $u_{p,s}$ we have
\begin{equation*}
u_{p,s}(q,t)\leq \frac{\bar{C}}{[4\pi(s-t)]^{n/2}}\;e^{-d^2_{g(t)}(p,q)/\bar{C}(s-t)},
\end{equation*}
for all $q\in M$ and $t\in[0,s)$. Hence, writing $\tau=T-t$,
\begin{equation}\label{Gaussupper}
\up(q,t)\leq  \frac{\bar{C}}{(4\pi\tau)^{n/2}}\;e^{-d^2_{g(t)}(p,q)/\bar{C}\tau},
\end{equation}
for every $\up\in\sU_p$ and $(q,t)\in M\times[0,T)$.
\end{prop}

\begin{proof}
In Theorem 1.30 of~\cite{HN12}, Hein and Naber prove that (in our notation)
\begin{equation*}
\nu_{p,s}(A)\nu_{p,s}(B)\leq \exp\Big({-\frac{1}{8(s-t)}}\mathrm{dist}^2_{g(t)}(A,B)\Big)
\end{equation*}
for all sets $A,B\subseteq M$, where $\mathrm{dist}_{g(t)}$ refers to the usual set distance and
\begin{equation*}
\nu_{p,s}(A):=\int_A u_{p,s}(q,t)dV_{g(t)}(q).
\end{equation*}
Note that this does not require any assumptions on the curvature or
the underlying manifold at all. We choose $A=B_{g(t)}(q,r)$,
$B=B_{g(t)}(p,r)$, with $r:=\sqrt{s-t}$. From Perelman's
non--collapsing at (scalar) curvature scale result (see~\cite{Per02},
or Theorem 2.20 in~\cite{HN12} for the precise version used here), we
know that there exists $\kappa>0$ depending on the Type~I constant
$C_{\II}$ and the value of the $\sW$--functional of the initial
manifold $(M,g(0))$, such that $\mathrm{vol}_{g(t)}(A),
\mathrm{vol}_{g(t)}(B)\geq \kappa r^n$. Hence, we obtain the estimate
\begin{equation*}
\frac{\nu_{p,s}(A)}{\mathrm{vol}_{g(t)}(A)}\leq \frac{1}{\kappa r^n}\;\frac{1}{\nu_{p,s}(B)}\;\exp\Big({-\frac{1}{C_1(s-t)}}d^2_{g(t)}(p,q)+C_1\Big)
\end{equation*}
for some $C_1>0$ depending only on $n$ and $C_{\II}$. Now, using the lower bound from Proposition~\ref{effectiveprop}, we have
\begin{equation*}
\nu_{p,s}(B)=\int_B u_{p,s}(q,t)\,dV_{g(t)}(q)\geq
\frac{\wh{C}e^{-1/\wh{C}}}{[4\pi(s-t)]^{n/2}}\;\mathrm{vol}_{g(t)}(B)\geq
\frac{\kappa\wh{C}e^{-1/\wh{C}}}{(4\pi)^{n/2}}.
\end{equation*}
Plugging this into the above, we find the average integral estimate
\begin{equation}\label{average}
\dashint_{B_{g(t)}(q,r)}\! u_{p,s}(q,t) dV_{g(t)}(q)\leq \frac{C_2}{\abs{s-t}^{n/2}}e^{-d^2_{g(t)}(p,q)/C_2(s-t)}.
\end{equation}
for some $C_2>0$ depending on $n$, $C_1$, $C_{\II}$ and $\kappa$ and thus ultimately on $n$, $C_{\II}$ and the initial manifold $(M,g(0))$. Doing the same for earlier time--slices too, we find a constant $C_3>0$ depending only on $n$, $C_{\II}$ and $(M,g(0))$, such that
\begin{equation*}
\dashint_{[t-r^2,t]}\dashint_{B_{g(t)}(q,r)}\! u_{p,s}(q,\lambda) dV_{g(t)}\, d\lambda \leq \frac{C_3}{\abs{s-t}^{n/2}}e^{-d^2_{g(t)}(p,q)/C_3(s-t)}.
\end{equation*}
The claim then follows from the parabolic mean value inequality, see
for example Theorem 25.2 in~\cite{RFTA3} for precisely the setting we
need here. Remember that the necessary lower Ricci bound on
$[t-r^2,t]\times B_{g(t)}(q,r)$ follows from the Type~I assumption.
\end{proof}

In fact, for the purpose of this paper, we only need the following much weaker growth property of the functions $\fp$, which is a direct consequence of the Gaussian type bounds above and which is strong enough to ensure that we are not ``losing'' $\sW$--entropy along the blow--up process. We prefer to work with this weak bound as there is hope to prove it under much weaker assumptions than the Type~I condition used here, while we do not expect the Gaussian upper bounds from Proposition~\ref{prop.Gaussian} to hold in the Type~II case.

\begin{cor}\label{rinf} 
Let $(M,g(t))$ be a compact Type~I Ricci flow on $[0,T)$. Then, there exists a constant $\bar{r}=\bar{r}(M,g_0)\in\RR^+$ independent
of $p\in M$, such that for every $\fp\in\sF_p$ there holds
$\fp(q,t)\geq 3n$ whenever $d^2_{g(t)}(p,q)\geq(T-t)\bar{r}^2$.
\end{cor}

\begin{proof}
Define $\bar{r}:=\big(3\bar{C}n+\bar{C}\log\bar{C})^{1/2}$ with $\bar{C}$ as in Proposition~\ref{prop.Gaussian}. Then $\frac{1}{\bar{C}}\bar{r}^2=3n+\log\bar{C}$, and Proposition~\ref{prop.Gaussian} implies
\begin{equation*}
\fp(q,t)\geq \frac{d^2_{g(t)}(p,q)}{\bar{C}(T-t)}-\log\bar{C}\geq \frac{1}{\bar{C}}\bar{r}^2-\log\bar{C}=3n.\qedhere
\end{equation*}
\end{proof}

\section{The $\sW$--Density Functions}\label{secW}

We define the $\sW$--density functions $\tp:[0,T)\to\RR$ as
modifications of Perelman's entropy function $\mu:[0,T)\to\RR$
(see~\cite{Per02}). Let $(M,g(t))$ be a compact Ricci flow on a finite
time interval $[0,T)$ and let $\sF_p$ be as in
Definition~\ref{defahk} above. For $\sF_p\ni\fp:M\times[0,T)\to\RR$,
we set $\tfpt(t):=\sW(g(t),\fp(t),T-t)$. Then the
\emph{$\sW$--density function} is defined as the infimum
\begin{align*}
\tp(t)&=\inf_{\fp\in\sF_p}\tfpt(t) = \inf_{\fp\in \sF_p}\sW(g(t),\fp(t),\tau)\\
&=\inf_{\fp\in \sF_p} \int_M
\Big(\tau (\RRR_{g(t)}+\abs{\nabla\fp(t)}^2_{g(t)})+\fp(t)-n\Big)
\frac{e^{-\fp(t)}}{(4\pi\tau)^{n/2}}dV_{g(t)},
\end{align*}
where $\tau=T-t$. We will often write $\theta(p,t)=\tp(t)$. Moreover, we also define the following {\em minimal  $\sW$--density function} $\lambda:[0,T)\to\RR$ as $\lambda(t)=\inf_{p\in M} \tp(t)$.

By the compactness property of the families $\sF_p$, these infima are
actually minima, that is, for every time $t\in[0,T)$ and point $p\in
M$ there is always a function in $\sF_p$ realizing the infimum in the
definition of $\tp(t)$. This means that $\fph(t)$ in
Definition~\ref{minimizer} is well--defined. The densities $\tfpt$,
$\tp$ and $\lambda$ are obviously uniformly bounded below by
Perelman's entropy function $\mu(t)=\sW(g(t),f_{\min}(t),T-t)$ 
(see the discussion at the beginning of Section~\ref{secfpT}). Let us
now prove the nonpositivity and monotonicity of these functions.

\begin{proof}[Proof of Proposition~\ref{monoprop}]
By Perelman's entropy formula in~\cite{Per02}, for every normalized function $f:M\times[0,T)\to\RR$ defined in such a way that $u(\cdot,t)=\frac{e^{-f(\cdot,t)}}{(4\pi\tau)^{n/2}}$ satisfies the adjoint heat equation $\dt u=-\Lap u+\RRR_g u$, we have along a Ricci flow $(M,g(t))$
\begin{equation}\label{permon}
\dt \sW(g(t),f(\cdot,t),\tau(t))= 2\tau\int_M \Abs{\Rc_{g(t)}+\nabla^2 f(\cdot,t)-\frac{g(t)}{2\tau}}^2_{g(t)}\frac{e^{-f(\cdot,t)}}{(4\pi\tau)^{n/2}}dV_{g(t)},
\end{equation}
for every $\tau(t)$ with $\dt \tau=-1$. Since these conditions are obviously satisfied by any $\fp\in\sF_p$ and $\tau=T-t$, we obtain the monotonicity formula 
\begin{equation}\label{mono1}
\dt \tfpt(t) = 2\tau\int_M \Abs{\Rc_{g(t)} + \nabla^2 \fp(\cdot,t) -\frac{g(t)}{2\tau}}_{g(t)}^2\frac{e^{-\fp(\cdot,t)}}{(4\pi\tau)^{n/2}}dV_{g(t)} \geq 0.
\end{equation}
This derivative is locally bounded in time, uniformly in $\fp\in\cup_{p\in M}\sF_p$ (by the estimates of Section~\ref{secfpT}) and hence the functions $\tp$ are uniformly locally
Lipschitz, thus differentiable at almost every time $t\in[0,T)$. It is then easy to see that formula~\eqref{monotonicity} holds at every differentiable time, where $\fph(t)$ is the minimizer in the definition of $\tp(t)$.

Finally, the nonpositivity of $\theta_p$ is a consequence of
Proposition~\ref{nilemma2}. Since for any $\fp$ we have
\begin{equation*}
(T-t)\big(2\Lap \fp(q,t)
-\abs{\nabla \fp(q,t)}^2_{g(t)}
+\RRR_{g(t)}(q,t)\big)+\fp(q,t)-n\leq 0,
\end{equation*}
at every point $q\in M$ and $t\in[0,T)$, the same must clearly hold
for all the functions $\fph(t)$ as well. Then, using integration by parts, we get
\begin{align*}
\tfpt(t)&=\int_M \Big(\tau (\RRR_g+\abs{\nabla\fp}^2_g)+\fp-n\Big)\frac{e^{-\fp}} {(4\pi\tau)^{n/2}}dV_g\\
&=\int_M \Big(\tau (\RRR_g+2\Lap\fp-\abs{\nabla\fp}^2_g)+\fp-n\Big)
\frac{e^{-\fp}}{(4\pi\tau)^{n/2}}dV_g\leq 0,
\end{align*}
and the same for $\theta_p(t)$,
\begin{align*}
\theta_p(t)&=\int_M \Big(\tau (\RRR_{g(t)}+\abs{\nabla\fph}^2_{g(t)}) +\fph- n\Big)\frac{e^{-\fph}}{(4\pi\tau)^{n/2}}dV_{g(t)}\\
&=\int_M\Big(\tau (\RRR_{g(t)}+2\Lap\fph-\abs{\nabla\fph}^2_{g(t)})+\fph-n\Big)
\frac{e^{-\fph}}{(4\pi\tau)^{n/2}}dV_{g(t)}\leq 0,
\end{align*}
for every $t\in[0,T)$. This finishes the proof.
\end{proof}

\begin{cor}[Limit $\sW$--density function]
For every $p\in M$, the function $\tp:[0,T)\to\RR$ converges to some
nonpositive value $\Theta(p)$ as $t\to T$. We call this value 
the {\em limit $\sW$--density} of the flow at the point $p\in
M$. The function $\lambda:[0,T)\to\RR$ is also monotone
non--decreasing and nonpositive and it converges to some value $\Lambda$ with
$\Lambda\leq\Theta(p)\leq 0$ for every $p\in M$.
\end{cor}

\begin{proof}
We have seen above that $\tp:[0,T)\to\RR$ is monotone non--decreasing
and nonpositive (as we will see in Section~\ref{secResc}, the density $\tp$ is
actually always negative), hence the convergence $\tp(t)\to\Theta(p)\leq 0$ 
as $t\to T$ is a trivial consequence. The monotonicity and
convergence of $\lambda$ is also trivial, being the infimum of
a family of nonpositive and monotone non--decreasing functions.
\end{proof}

Integrating the entropy formula~\eqref{monotonicity} in time and setting 
$\uph(t):=\frac{e^{-\fph(t)}}{(4\pi\tau)^{n/2}}$, we get for
every $t_0,t_1\in[0,T)$ with $t_0<t_1$,
\begin{equation*}
\tp(t_1)-\tp(t_0)=2\int_{t_0}^{t_1}\tau\int_M
\Abs{\Rc_{g(t)}+\nabla^2\fph(t)-\frac{g(t)}{2\tau}}^2_{g(t)}\uph(t)
dV_{g(t)}\,dt,
\end{equation*}
and passing to the limit as $t_1\to T$,
\begin{equation}\label{perelman3}
\Theta(p)-\tp(t_0)=2\int_{t_0}^T\tau\int_M
\Abs{\Rc_{g(t)}+\nabla^2\fph(t)-\frac{g(t)}{2\tau}}^2_{g(t)}\uph(t)
dV_{g(t)}\,dt.
\end{equation}
Hence, as $\Theta(p)\leq0$, we have
\begin{equation}\label{perelman4}
\int_0^T\tau\int_M
\Abs{\Rc_{g(t)}+\nabla^2\fph(t)-\frac{g(t)}{2\tau}}^2_{g(t)}\uph(t)
dV_{g(t)}\,dt\leq -\frac{\theta_p(0)}{2}\leq C(M,g_0,T),
\end{equation}
where the constant $C(M,g_0,T)$ is given by
\begin{equation*}
C(M,g_0,T)=-\frac{1}{2}\inf_{p\in M}\inf_{\fp\in \sF_p} \int_M \Big(T\big(\RRR_{g_0}+\abs{\nabla\fp}^2_{g_0}\big)+\fp-n\Big)\frac{e^{-\fp}}{(4\pi T)^{n/2}}\,dV_{g_0}
\end{equation*}
which is finite due to the compactness property of the union $\sF=\cup_{p\in M}\sF_p$. Of course, estimate~\eqref{perelman4} also holds if we replace the minimizer $\fph$ with any choice of $\fp\in\sF_p$ and use formula~\eqref{mono1} instead of~\eqref{monotonicity}.

The main reason for working with $\tp$ instead of $\tfpt$ is its continuous dependence on the point $p\in M$.

\begin{lemma}[Continuous dependence of $\theta_p(t)$ on $p\in M$]\label{thcon}
The $\sW$--density functions $\theta(\cdot,t):M\to\RR$ are continuous for every fixed $t\in[0,T)$. The limit $\sW$--density function $\Theta:M\to\RR$ is lower semicontinuous and nonpositive, hence every point $p\in M$ with $\Theta(p)=0$ is a continuity point.
\end{lemma}
\begin{proof}
The first statement is a straightforward consequence of the
compactness of the union $\sU=\cup_{p\in M}\sU_p$ in the
$C^\infty_{\mathrm{loc}}(M\times[0,T))$--topology, see
Lemma~\ref{compactlemma}. Because all the functions $\theta(\cdot,t)$ are continuous in $p\in M$, the limit function $\Theta:M\to\RR$ is lower semicontinuous, the final claim is then trivial.
\end{proof}

A consequence is the following corollary.

\begin{cor}
The function $\Theta:M\to\RR$ is identically zero on $M$ if and only
if $\Lambda=0$. In other words, if and only if the functions
$\theta(\cdot,t):M\to\RR$ converge uniformly to zero as $t\to T$.
\end{cor}
\begin{proof}
If $\Theta$ is identically zero, by Dini's monotone convergence
theorem, the functions $\theta(\cdot,t)$ uniformly converge to zero as
$t\to T$ and hence $\Lambda=0$. The other implication is trivial.
\end{proof}

\section{Gradient Shrinking Ricci Solitons and their $\sW$--Entropy}\label{secShrinker}

Let $(\Minf,\ginf,\finf)$ be a gradient shrinking Ricci soliton, that
is, a complete, connected Riemannian manifold $(\Minf,\ginf)$ satisfying the relation
\begin{equation*}
\Rc_{\ginf}+\nabla^2\finf=\frac{\ginf}{2},
\end{equation*}
where $\finf:\Minf\to\RR$ is a smooth function.

It is well known that the quantity $a(\ginf,\finf):=\RRR_{\ginf}
+\abs{\nabla\finf}^2_{\ginf}-\finf$ is constant on $M$, it is often called 
{\em auxiliary constant}.

We recall the following growth estimates, originally proved by
Cao--Zhou and Munteanu~\cite{CZ10,Mun09} and improved by 
Haslhofer--M\"{u}ller~\cite{HM11} to the present form.

\begin{lemma}[Potential and volume growth, Lemma~2.1 and~2.2 in~\cite{HM11}]\label{growth}
Let $(\Minf,\ginf,\finf)$ be an $n$--dimensional gradient shrinking Ricci soliton with auxiliary constant $a(\ginf,\finf)$. Then there exists a point $\pinf\in\Minf$ where
$\finf$ attains its infimum and we have the following estimates for
the growth of the potential
\begin{equation*}
\frac{1}{4}\big(d_{\ginf}(x,\pinf)-5n\big)_{\!+}^2\leq \finf(x)-a(\ginf,\finf) \leq
\frac{1}{4}\big(d_{\ginf}(x,\pinf)+\sqrt{2n}\big)^2.
\end{equation*}
Moreover, we have the volume growth estimate $\mathrm{Vol}(B_r^\infty(\pinf))\leq V(n)r^n$ for geodesic balls in $(\Minf,\ginf)$ around $\pinf\in\Minf$, where $V(n)$ is a constant depending only on the dimension $n$ of the soliton.
\end{lemma}

As a consequence of these estimates, $\int_{\Minf} e^{-\finf}dV_{\ginf}$ is well--defined and we can always normalize the potential function $\finf$ by adding a constant in such a way that
\begin{equation}\label{normalization}
\int_{\Minf} \frac{e^{-\finf}}{(4\pi)^{n/2}}\,dV_{\ginf} = 1.
\end{equation}
We then call such a potential function $\finf$ and the resulting 
soliton $(\Minf,\ginf,\finf)$ \emph{normalized}.

Lemma~\ref{growth} implies that every function $\phi$ satisfying
$\abs{\phi(x)}\leq Ce^{\alpha d_{\ginf}^2(x,\pinf)}$ for some
$\alpha<\frac{1}{4}$ and constant $C$, is integrable with respect to
$e^{-\finf}dV_{\ginf}$. In particular, since $0\leq \RRR_{\ginf}+\abs{\nabla\finf}^2_{\ginf}\leq
\finf+a(\ginf,\finf)$ and $\Lap \finf= \frac{n}{2}-\RRR_{\ginf}$, this holds for every 
polynomial in $\finf$, $\abs{\nabla\finf}^2_{\ginf}$, $\RRR_{\ginf}$ and
$\Lap\finf$. Hence, every gradient shrinking Ricci soliton 
has a well--defined $\sW$--\emph{entropy}
\begin{equation}\label{entropy}
\sW(\ginf,\finf):=\int_{\Minf} \big(\RRR_{\ginf}
+\abs{\nabla \finf}^2_{\ginf}+\finf-n\big)
\frac{e^{-\finf}}{(4\pi)^{n/2}}\,dV_{\ginf}.
\end{equation}

Let us now collect some properties of shrinking solitons and their
$\sW$--entropy that we will use in the next section.
\begin{lemma}\label{solprop}
For every gradient shrinking Ricci soliton $(\Minf,\ginf)$ with
potential function $\finf:\Minf\to\RR$, the following properties holds:
\begin{enumerate}
\item Either the scalar curvature $\RRR_{\ginf}$ is positive everywhere or $(\Minf,\ginf)$ is the standard flat $\RR^n$, that is, $(\Minf,\ginf,\finf)$ is the Gaussian soliton.

\item\label{sol2} There holds
\begin{equation*}
\sW(\ginf,\finf)
=\int_{\Minf} \big(\RRR_{\ginf}+2\Lap \finf
-\abs{\nabla\finf}^2_{\ginf}+\finf-n\big)\frac{e^{-\finf}}{(4\pi)^{n/2}}\,dV_{\ginf}.
\end{equation*}

\item The $\sW$--entropy $\sW(\ginf,\finf)$ is equal to $-a(\ginf,\finf)\int_{\Minf} \frac{e^{-\finf}}{(4\pi)^{n/2}}\,dV_{\ginf}$.

\item Any two normalized potential functions $\finf^1$ and $\finf^2$ of the same soliton $(\Minf,\ginf)$ share the same auxiliary constant, that is $a(\ginf,\finf^1)=a(\ginf,\finf^2)$. Hence, $\sW(\ginf)=\sW(\ginf,\finf)$ is independent of the choice of the normalized potential function $\finf$.

\item\label{sol4} If the Ricci tensor of a normalized soliton $(\Minf,\ginf,\finf)$ is bounded below, we have $\sW(\ginf)\leq 0$ and if $\sW(\ginf)=0$ then the manifold $(\Minf,\ginf)$ is the flat $\RR^n$ (Gaussian soliton). Moreover, under the same hypotheses, we have the following gap result: there exists a dimensional constant $\eps_n>0$ such that if $\sW(\ginf)\neq 0$, then $\sW(\ginf)<-\eps_n$.
\end{enumerate}
\end{lemma}

\begin{proof}\ 
\begin{enumerate}
\item This is a result of Zhang~\cite[Theorem~1.3]{Zha09} and Yokota~\cite[Appendix~A.2]{Yok09} (see also~Pigola, Rimoldi and Setti~\cite{PRS11}).

\item The necessary partial integration formula 
\begin{equation*}
\int_{\Minf}\Lap\finf e^{-\finf}dV_{\ginf}=\int_{\Minf}\abs{\nabla\finf}^2_{\ginf} e^{-\finf}dV_{\ginf}
\end{equation*}
follows from the growth estimates of Lemma~\ref{growth} using a cut--off argument. See Section~2 of Haslhofer--M\"{u}ller~\cite{HM11} for full detail.

\item By the auxiliary equation $a(\ginf,\finf)=\RRR_{\ginf} +\abs{\nabla\finf}^2_{\ginf}-\finf$ and the traced soliton equation $\RRR_{\ginf}+\Lap\finf=\frac{n}{2}$, we have
\begin{equation*}
\RRR_{\ginf}+2\Lap \finf-\abs{\nabla \finf}^2_{\ginf}+\finf-n=-a(\ginf,\finf)
\end{equation*}
and thus $\sW(\ginf,\finf)=-a(\ginf,\finf)\,\int_{\Minf} \frac{e^{-\finf}}{(4\pi)^{n/2}}\,dV_{\ginf}$ follows from Point~\ref{sol2}.

\item Since the Hessian of any potential of the soliton 
is uniquely determined by the soliton equation, the difference function $h:=\finf^1-\finf^2$ is either a constant or the vector field $\nabla h$ is parallel. 
In the first case, the constant has to be zero by the normalization condition~\eqref{normalization}. In the second case, by de~Rham's splitting theorem, $(\Minf,\ginf)$ isometrically splits off a line (see for instance~\cite[Theorem~1.16]{CLN06}). Hence, we let $(\Minf,\ginf)=(\widetilde{M}_{\infty},\widetilde{g}_{\infty})\times(\RR^k,\mathrm{can})$, with $1\leq k\leq n$, such that $\widetilde{M}_\infty$ cannot split off
a line.

Denoting by $x$ the coordinates on $\widetilde{M}_\infty$ and by $y$ the coordinates on $\RR^k$, the soliton equation implies that both potentials also split as $\finf^\ell(x,y)=\widetilde{f}_{\infty}^\ell(x)+\frac{1}{4}\abs{y-y_\ell}^2_{\RR^k}$ for $\ell=1,2$, where $y_\ell\in\RR^k$. Moreover, $(\widetilde{M}_\infty,\widetilde{g}_\infty)$ is still a gradient shrinking Ricci soliton with both functions 
$\widetilde{f}_{\infty}^\ell:\widetilde{M}_\infty\to\RR$ as possible potentials, and since $\widetilde{M}_\infty$ cannot split off a line, they must differ by a constant. Thus, we have
\begin{equation*}
\finf^\ell(x,y)=\widehat{f}_{\infty}(x)+\alpha_\ell+\tfrac{1}{4}\abs{y-y_\ell}^2_{\RR^k},
\end{equation*}
for some function $\widehat{f}_{\infty}:\widetilde{M}_\infty\to\RR$ and
two constants $\alpha_1$ and $\alpha_2$.

Now, integrating the two functions $e^{-\finf^\ell}$, by means of Fubini's theorem and the normalization condition~\eqref{normalization}, we conclude that $\alpha_\ell=\alpha$ and we obtain
\begin{equation*}
a(\ginf,\finf^\ell)=\RRR_{\ginf} + \abs{\nabla\finf^\ell}^2_{\ginf} -\finf^\ell=
\RRR_{\ginf} +\abs{\nabla\widehat{f}_{\infty}}^2_{\ginf}-\widehat{f}_{\infty}-\alpha,
\end{equation*}
which is independent of $\ell=1,2$.

\item This point is a result of Yokota (Carillo--Ni~\cite{CN08} got similar results under more restrictive curvature hypotheses). Our version is equivalent to his statement~\cite[Corollary~1.1]{Yok09}.\qedhere
\end{enumerate}
\end{proof}

\begin{rem} 
If we restrict ourselves to deal with gradient shrinking Ricci solitons coming from a blow--up of a compact manifold (which is sufficient for the aims of this paper), we do not need the general result at Point~1 of this lemma, since every soliton we obtain after rescaling must have $\RRR_{\ginf}\geq 0$ everywhere, by the well known uniform bound from below on the scalar curvature, $\RRR^{\min}_{g(t)}\geq\RRR^{\min}_{g(0)}$ for every $t\in[0,T)$. Then, by a standard strong maximum principle argument (see~\cite{PRS11}, for instance), if $\RRR_{\ginf}$ is zero somewhere, the soliton must be flat, hence the Gaussian soliton.

We underline that in our situation also the nonpositivity of the $\sW$--entropy at Point~5 follows by the construction, using Theorem~\ref{noloss} and Corollary~\ref{thetaprop}, but no assumption on the Ricci tensor.

It is unknown to the authors whether the family of gradient shrinking Ricci solitons coming from a blow--up of a compact Ricci flow coincides or not with the full class of gradient shrinking Ricci solitons. For instance, by Perelman's work the former must be 
non--collapsed, we do not know if all the general shrinkers satisfy this condition.
\end{rem}

In the cases where a full classification of the shrinkers is possible,
for instance in low dimensions (two and three), it turns out that two
different gradient shrinking Ricci solitons cannot share the same
value of the $\sW$--entropy. This motivates the following definition.
\begin{defn}\label{defunique}
A normalized gradient shrinking Ricci soliton is called \emph{entropy--unique} if any other normalized shrinker with the same value of the entropy is isometric to it.
\end{defn}

\section{The Proofs of the Main Theorems}\label{secResc}

\begin{proof}[Proof of Theorem~\ref{mainthm} --- Dynamical blow--up]
Let $\sF_p$ be as in Definition~\ref{defahk} and let
$\fph(t)$ be the minimizer for $\sW$ among all functions
$\fp(t)\in\sF_p$ as in Definition~\ref{minimizer} with corresponding
$\uph=\frac{e^{-\fph}}{[4\pi(T-t)]^{n/2}}$. 

We first study the dynamical blow--up $(M,\wt{g}(s),\wt{f}(s),p)$ with
\begin{equation*}
\wt{g}(s)=\frac{g(t)}{T-t}\,\,\,\text{ and }\,\,s=-\log{(T-t)}.
\end{equation*}
We set $\wt{f}(s)=\fph(t)$. Rescaling the integral formula~\eqref{perelman4}, we get
\begin{align*}
C(M,g_0,T) &\geq
\int_0^T\tau\int_M\Abs{\Rc_{g(t)}+\nabla^2\fph(t)-\frac{g(t)}{2\tau}}^2_{g(t)}
\uph(t)\,dV_{g(t)}\,dt\\
&=\int_{-\log{T}}^{+\infty}\int_M\Abs{\Rc_{\wt{g}(s)}+\nabla^2\wt{f}(s)
-\frac{\wt{g}(s)}{2}}^2_{\wt{g}(s)}
\frac{e^{-\wt{f}(s)}}{(4\pi)^{n/2}}\,dV_{\wt{g}(s)}\,ds.
\end{align*}
Note that this implies the formula~\eqref{L2eqdyn}. Furthermore, it 
follows that for every family of disjoint intervals $(a_k,b_k)\subset\RR$ with
$\sum_{k\in\NN} (b_k-a_k)=+\infty$, we have a sequence of times
$s_j\in\cup_{k\in\NN}(a_k,b_k)$ with $s_j\nearrow\infty$ such that
\begin{equation*}
\int_M\Abs{\Rc_{\wt{g}(s_j)}+\nabla^2\wt{f}(s_j)-\frac{\wt{g}(s_j)}{2}}^2_{\wt{g}(s_j)}
\frac{e^{-\wt{f}(s_j)}}{(4\pi)^{n/2}}\,dV_{\wt{g}(s_j)}\to 0.
\end{equation*}
In the Type~I case, we have uniform bounds on the rescaled curvatures
$\abs{\Rm_{\wt{g}(s)}}_{\wt{g}(s)}$, on the functions $\wt{f}(s)$ and
their covariant derivatives, and on the injectivity radii (due to
Perelman's non--collapsing theorem in~\cite{Per02}). Thus, the rescaled pointed
manifolds $(M,\wt{g}(s_j),p)$ converge (up to a subsequence) in the
pointed Cheeger--Gromov--Hamilton sense~\cite{Ham95} to a complete
smooth limit Riemannian manifold $(\Minf,\ginf,\pinf)$ and the
functions $\wt{f}(s_j):M\to\RR$ converge locally smoothly to some
smooth function $\finf:\Minf\to\RR$. It remains to show
that the limit $(\Minf,\ginf,\finf)$ is a normalized gradient shrinking Ricci soliton.

By Proposition~\ref{effectiveprop}, there exists a positive constant
$\wh{C}$, independent of $i\in\NN$, such that 
\begin{equation}\label{endofproof1}
e^{-\wt{f}(q,s_i)}\geq \wh{C}e^{-d^2_{\wt{g}(s_j)}(p,q)/\wh{C}}.
\end{equation}
Hence, in every geodesic ball $\widetilde{B}^j_\varrho$ in $(M,\wt{g}(s_j))$ of radius
$\varrho>0$ around the basepoint $p\in M$, we have 
\begin{equation*}
\int_{\widetilde{B}_\varrho^j}\Abs{\Rc_{\wt{g}(s_j)}+\nabla^2 \wt{f}(s_j)-\frac{\wt{g}(s_j)}{2}}^2_{\wt{g}(s_j)}
e^{-d^2_{\wt{g}(s_j)}(p,q)/\wh{C}}\,dV_{\wt{g}(s_j)}\to 0,
\end{equation*}
thus
\begin{equation*}
\int_{\widetilde{B}_\varrho^j}\Abs{\Rc_{\ginf}+\nabla^2\finf-\frac{\ginf}{2}}^2_{\ginf}
e^{-d^2_{\ginf}(p,q)/\wh{C}}\,dV_{\ginf}=0
\end{equation*}
and, since $\varrho$ was arbitrary,
\begin{equation}\label{endofproof2}
\Rc_\infty+\nabla^2 \finf=\frac{\ginf}{2}
\end{equation}
everywhere on $\Minf$.

To finish the proof of Point~1 of Theorem~\ref{mainthm}, we only need to show that the limit gradient shrinking Ricci soliton $(\Minf,\ginf,\finf)$ is normalized. This will be done in Corollary~\ref{convnorm} below.
\end{proof}

\begin{proof}[Proof of Theorem~\ref{mainthm} --- Sequential blow--up]
The Type~I assumption~\eqref{TypeIupper} translates to the uniform curvature bound
\begin{equation*}
\sup_M\abs{\Rm(\cdot,s)}_{g_j(s)}=\sup_M\frac{1}{\lambda_j}
\abs{\Rm(\cdot,T+\tfrac{s}{\lambda_j})}_{g(T+\frac{s}{\lambda_j})}\leq
\frac{C}{\lambda_j\big(T-(T+\frac{s}{\lambda_j})\big)}=\frac{C}{-s}
\end{equation*}
for the blow--up sequence $(M,g_j(s),p)$ defined in formula~\eqref{sequential}. This yields uniform curvature bounds on compact time intervals $[S_0,S_1]\subset(-\infty,0)$ using Bando--Shi estimates~\cite{Shi89}. Together with Perelman's no local collapsing theorem, we can again use the Cheeger--Gromov--Hamilton compactness theorem~\cite{Ham95} to extract a complete pointed subsequential limit Ricci flow $(\Minf,\ginf(s),\pinf)$ on $(-\infty,0)$ which is still Type~I. This means in particular, that there exist an exhaustion of $\Minf$ by open sets $U_j$ containing $\pinf$ and smooth embeddings $\phi_j: U_j\to M$ with $\phi_j(\pinf)=p$ such that the pulled back metrics $\phi_j^*g_j$ converge to $\ginf$ smoothly on compact subsets of $\Minf\times(-\infty,0)$. Using the uniform bounds for $f_j(s)=\fph(t)$ (where $s=\lambda_j(t-T)$) from Section~\ref{secfpT}, we see that (by possibly taking a further subsequence) also the pull--backs $\phi_j^*f_j$ converge smoothly to a limit $\finf:\Minf\times(-\infty,0)\to\RR$.

For $[S_0,S_1]\subset(-\infty,0)$ and $g_j$, $f_j$ as above, we compute, using $\tau=T-t$ and the monotonicity formula~\eqref{monotonicity},
\begin{align*}
0&=\lim_{j\to\infty}\tfpt(g(T+\frac{S_1}{\lambda_j})) -
\lim_{j\to\infty}\tfpt(g(T+\frac{S_0}{\lambda_j}))\\
&=\lim_{j\to\infty}\int_{T+\frac{S_0}{\lambda_j}}^{T+\frac{S_1}{\lambda_j}}
2\tau(4\pi\tau)^{-n/2}\int_M\Abs{\Rc_{g(t)}+\nabla^2\fph(t)-\frac{g(t)}{2\tau}}^2_{g(t)}e^{-\fph(t)}dV_{g(t)}\,dt\\
&=\lim_{j\to\infty}\int_{S_0}^{S_1} -2s(-4\pi
s)^{-n/2}\int_M\Abs{\Rc_{g_j(t)}+\nabla^2
f_j(s)-\frac{g_j(s)}{-2s}}^2_{g_j(s)}e^{-f_j(s)}dV_{g_j(s)}\,ds.
\end{align*}
Since $-2s(-4\pi s)^{-n/2}$ is positive and bounded on $[S_0,S_1]$,
this implies~\eqref{L2eqseq}. The soliton property of the limit then
follows from~\eqref{L2eqseq} analogous
to~\eqref{endofproof1}--\eqref{endofproof2}. As above, the proof of
Point~2 of Theorem~\ref{mainthm} is finished by arguing that the limit soliton is normalized, which follows from Corollary~\ref{convnorm} below.
\end{proof}

\begin{rem}
All the formulas hold true if we replace the minimizer $\fph(t)$ with
some fixed choice of $\fp(t)\in\sF_p$.
\end{rem}

Next, we relate the limit $\sW$--density $\Theta(p)$ with the value of the $\sW$--entropy
$\sW(\ginf,\finf)$ of the limit gradient shrinking soliton $(\Minf,\ginf,\finf)$ obtained by the above rescaling procedures. The proof of Theorem~\ref{noloss} presented here uses only the weak bounds for $\fp$ from Corollary~\ref{rinf} (rather than the full Gaussian upper bounds from Proposition~\ref{prop.Gaussian}) and it also does not rely on the fact that the limit shrinkers are normalized (since we want to prove this as a corollary below).

\begin{proof}[Proof of Theorem~\ref{noloss}]
Using integration by parts, for every $j\in\NN$, we have 
\begin{align*}
\int_{M}\big(\RRR_{\wt{g}(s_j)}&+\abs{\nabla\wt{f}(s_j)}^2_{\wt{g}(s_j)}
+\wt{f}(s_j)-n\big)\,\frac{e^{-\wt{f}(s_j)}}{(4\pi)^{n/2}}\,dV_{\wt{g}(s_j)}\\
&=\int_{M}\big(\RRR_{\wt{g}(s_j)}+2\Lap\wt{f}(s_j)-\abs{\nabla\wt{f}(s_j)}^2_{\wt{g}(s_j)} 
+\wt{f}(s_j)-n\big)\frac{e^{-\wt{f}(s_j)}}{(4\pi)^{n/2}}\,dV_{\wt{g}(s_j)},
\end{align*}
the latter integral having the property that its integrand is pointwise nonpositive, 
as shown in Proposition~\ref{nilemma2}. This clearly implies that it is upper
semicontinuous on the sequence of Riemannian manifolds and functions converging locally smoothly, and combining this with Point~\ref{sol2} of Lemma~\ref{solprop} we find
\begin{align*}
\lim_{j\to\infty}\int_{M}\big(&\RRR_{\wt{g}(s_j)}
+\abs{\nabla\wt{f}(s_j)}^2_{\wt{g}(s_j)}
+\wt{f}(s_j)-n\big)\,\frac{e^{-\wt{f}(s_j)}}{(4\pi)^{n/2}}\,dV_{\wt{g}(s_j)}\\
&\leq\int_{\Minf} \big(\RRR_{\ginf}+2\Lap \finf
-\abs{\nabla \finf}^2_{\ginf} +\finf-n\big)\frac{e^{-\finf}}
{(4\pi)^{n/2}}\,dV_{\ginf}=\sW(\ginf,\finf).
\end{align*}

In order to prove the opposite inequality (i.e. to show that we do not lose $\sW$--entropy in the limit), consider the value $\bar{r}>0$ given by Corollary~\ref{rinf} and denote by $\widetilde{B}_{\bar{r}}^j$ the geodesic ball of radius $\bar{r}$ in
$(M,\widetilde{g}(s_j))$ around $p_j\in M$. Then, split the integral
\begin{align*}
\int_{M}\big(\RRR_{\wt{g}(s_j)}
+\abs{&\nabla\wt{f}(s_j)}^2_{\wt{g}(s_j)}
+\wt{f}(s_j)-n\big)\,\frac{e^{-\wt{f}(s_j)}}{(4\pi)^{n/2}}\,dV_{\wt{g}(s_j)}\\
&=\int_{M\setminus\widetilde{B}_{\bar{r}}^j}
\big(\RRR_{\wt{g}(s_j)} +\abs{\nabla\wt{f}(s_j)}^2_{\wt{g}(s_j)}
+\wt{f}(s_j)-n\big)\,\frac{e^{-\wt{f}(s_j)}}{(4\pi)^{n/2}}\,dV_{\wt{g}(s_j)}\\
&\quad+\int_{\widetilde{B}_{\bar{r}}^j}\big(\RRR_{\wt{g}(s_j)} +\abs{\nabla\wt{f}(s_j)}^2_{\wt{g}(s_j)}
+\wt{f}(s_j)-n\big)\,\frac{e^{-\wt{f}(s_j)}}{(4\pi)^{n/2}}\,dV_{\wt{g}(s_j)}
\end{align*}
and notice that the last integral converges, by the hypotheses, to
\begin{equation*}
\int_{{B}_{\bar{r}}^\infty(p_\infty)}\big(\RRR_{\wt{g}_\infty}
+\abs{\nabla\wt{f}_\infty}^2_{\wt{g}_\infty}
+\wt{f}_\infty-n\big)\,\frac{e^{-\wt{f}_\infty}}{(4\pi)^{n/2}}\,dV_{\wt{g}_\infty}.
\end{equation*}
We claim that the first integral in the sum has a nonnegative integrand when
$j\in\NN$ is large enough. Indeed, the minimum of the scalar curvature
along the Ricci flow is non--decreasing hence it is bounded below
uniformly by $\min_M \RRR_{g(0)}$ (we recall that $M$ is compact), this implies that
\begin{equation*}
\liminf_{j\in\infty}\min_M \RRR_{\wt{g}(s_j)}= 
\liminf_{j\in\infty}\min_M \RRR_{g(t_j)}(T-t_j)\geq 
\liminf_{j\in\infty}\min_M \RRR_{g(0)}(T-t_j)=0,
\end{equation*}
in particular, when $j$ is large, $\RRR_{\wt{g}(s_j)}\geq -n$. Now, by
Corollary~\ref{rinf} we have that $\widetilde{f}(s_j)(q)=\widehat{f}_{p_j,T}(q,t(s_j))\geq 3n$ if $d^2_{g(t(s_j))}(p_j,q)\geq(T-t(s_j))\bar{r}^2$, that is, when $d^2_{\wt{g}(s_j)}(p_j,q)\geq\bar{r}^2$. This
last condition is clearly satisfied if $q\in M\setminus \widetilde{B}_{\bar{r}}^j$, hence in such a case
\begin{equation*}
\RRR_{\wt{g}(s_j)}(q) +\abs{\nabla\wt{f}(s_j)(q)}^2_{\wt{g}(s_j)}
+\wt{f}(s_j)(q)-n\geq -n+\abs{\nabla\wt{f}(s_j)(q)}^2_{\wt{g}(s_j)}+3n-n\geq n>0,
\end{equation*}
as we claimed.

Then, on the sequence of Riemannian manifolds and functions converging locally
smoothly, this integral is lower semicontinuous, that is,
\begin{align*}
\lim_{j\to\infty}\int_{M\setminus\widetilde{B}_{\bar{r}}^j}
\big(\RRR_{\wt{g}(s_j)} +&\,\abs{\nabla\wt{f}(s_j)}^2_{\wt{g}(s_j)}
+\wt{f}(s_j)-n\big)\,\frac{e^{-\wt{f}(s_j)}}{(4\pi)^{n/2}}\,dV_{\wt{g}(s_j)}\\
\geq&\,\int_{M_\infty\setminus{B}_{\bar{r}}^\infty(p_\infty)}\big(\RRR_{\wt{g}_\infty}
+\abs{\nabla\wt{f}_\infty}^2_{\wt{g}_\infty} 
+\wt{f}_\infty-n\big)\,\frac{e^{-\wt{f}_\infty}}{(4\pi)^{n/2}}\,dV_{\wt{g}_\infty}.
\end{align*}
Thus, putting together the limits of the two integrals and recalling the
definition of $\sW$--entropy $\sW(\ginf,\finf)$ in formula~\eqref{entropy}, 
we conclude
\begin{equation*}
\lim_{j\to\infty}\int_{M}\big(\RRR_{\wt{g}(s_j)}
+\abs{\nabla\wt{f}(s_j)}^2_{\wt{g}(s_j)}
+\wt{f}(s_j)-n\big)\,\frac{e^{-\wt{f}(s_j)}}{(4\pi)^{n/2}}\,dV_{\wt{g}(s_j)}\geq\sW(\ginf,\finf).\qedhere
\end{equation*}
\end{proof}

We can now give a proof of the fact that the limit solitons obtained above are normalized, thus finishing the proof of Theorem~\ref{mainthm}.

\begin{cor}\label{convnorm}
Given a sequence of pointed rescaled manifolds $(M,\wt{g}(s_j),p_j)$
and functions $\wt{f}(s_j)=\widehat{f}_{p_j,T}(t(s_j))$
converging locally smoothly to some gradient shrinking Ricci soliton
$(\Minf,\ginf,\pinf)$ and potential function $\finf:\Minf\to\RR$, we have
\begin{equation*}
\int_{\Minf}\frac{e^{-\finf}}{(4\pi)^{n/2}}\,dV_{\ginf}=1,
\end{equation*}
(that is, the soliton is normalized) if we are in one of the following two situations:
\begin{itemize}
\item either all the points $p_j$ coincide with some $p\in M$, like in Theorem~\ref{mainthm},
\item or if $p_j\to p$ and $\Theta(p)=0$.
\end{itemize}
\end{cor}

\begin{proof}
Since $\int_{M}\frac{e^{-\wt{f}(s_j)}}{(4\pi)^{n/2}}\,dV_{\wt{g}(s_j)}=1$
for every $j\in\NN$ and the functional is lower semicontinuous, as the
integrand is positive, it is sufficient to show
\begin{equation*}
\int_{\Minf}\frac{e^{-\finf}}{(4\pi)^{n/2}}\,dV_{\ginf}\geq 1.
\end{equation*}
Assume by contradiction that $\int_{\Minf}\frac{e^{-\finf}}
{(4\pi)^{n/2}}\,dV_{\ginf}=1-\alpha$ for some $\alpha>0$. This means
that for every $\varrho>0$ there exists $j_\varrho\in N$ such that for every
$j>j_\varrho$ we have
\begin{equation*}
\int_{M\setminus \widetilde{B}^j_\varrho(p_j)}\frac{e^{-\wt{f}(s_j)}}{(4\pi)^{n/2}}\,dV_{\wt{g}(s_j)}>\alpha/2.
\end{equation*}
If we take $\varrho$ larger than the value $\bar{r}$ given by Corollary~\ref{rinf} we have
\begin{equation*}
\int_{M\setminus \widetilde{B}^j_\varrho(p_j)}\big(\RRR_{\wt{g}(s_j)}+\abs{\nabla\wt{f}(s_j)}^2_{\wt{g}(s_j)} +\wt{f}(s_j)-n\big)\,\frac{e^{-\wt{f}(s_j)}}{(4\pi)^{n/2}}\,dV_{\wt{g}(s_j)}\geq \frac{n\alpha}{2}.
\end{equation*}
Hence,
\begin{equation*}
\int_{\widetilde{B}^j_\varrho(p_j)}\big(\RRR_{\wt{g}(s_j)}+\abs{\nabla\wt{f}(s_j)}^2_{\wt{g}(s_j)} +\wt{f}(s_j)-n\big)\,\frac{e^{-\wt{f}(s_j)}}{(4\pi)^{n/2}}\,dV_{\wt{g}(s_j)}\leq \theta(p_j,t(s_j))-\frac{n\alpha}{2},
\end{equation*}
and passing to the limit as $j\to\infty$ we get
\begin{equation*}
\int_{B^\infty_\varrho(\pinf)}\big(\RRR_{\ginf}+\abs{\nabla\finf}^2_{\ginf}
+\finf-n\big)\, \frac{e^{-\finf}}{(4\pi)^{n/2}}\,dV_{\ginf}\leq \lim_{j\to\infty}\theta(p_j,t(s_j))-\frac{n\alpha}{2}.
\end{equation*}
If all the points $p_j$ coincide with $p$, then $\lim_{j\to\infty}\theta(p_j,t(s_j))=\Theta(p)$, while in the second case, 
if $\Theta(p)=0$ we have the same conclusion by the fact that $p$ is a continuity point 
for $\Theta:M\to\RR$ by Lemma~\ref{thcon}. Hence,
\begin{equation*}
\int_{B^\infty_\varrho(\pinf)}\big(\RRR_{\ginf}+\abs{\nabla\finf}^2_{\ginf}
+\finf-n\big)\,\frac{e^{-\finf}}{(4\pi)^{n/2}}\,dV_{\ginf}\leq \Theta(p)-\frac{n\alpha}{2}
\end{equation*}
and letting $\varrho\to+\infty$ we get
\begin{equation*}
\int_{\Minf}\big(\RRR_{\ginf}+\abs{\nabla\finf}^2_{\ginf}
+\finf-n\big)\,\frac{e^{-\finf}}{(4\pi)^{n/2}}\,dV_{\ginf}=\sW(\ginf,\finf)\leq \Theta(p)-\frac{n\alpha}{2}
\end{equation*}
which is in contradiction with the conclusion of Theorem~\ref{noloss}.
\end{proof}

\begin{cor}\label{thetaprop}
We have $\sW(\ginf)=\Theta(p)$ for every normalized gradient shrinking
Ricci soliton obtained as a blow--up limit of a locally converging sequence 
of rescaled manifolds $(M,\wt{g}(s_j),p)\to (\Minf,\ginf,\pinf)$ and functions
$\wt{f}(s_j)=\widehat{f}_{p,T}(t(s_j))\to\finf$.
\end{cor}
\begin{proof}
For every rescaled manifold $(M,\wt{g}(s_j))$ we have, 
\begin{equation*}
\int_{M}\big(\RRR_{\wt{g}(s_j)} +\abs{\nabla\wt{f}(s_j)}^2_{\wt{g}(s_j)}
+\wt{f}(s_j)-n\big)\,\frac{e^{-\wt{f}(s_j)}}{(4\pi)^{n/2}}\,dV_{\wt{g}(s_j)}=\tp(t(s_j)).
\end{equation*}
Hence, the claim follows by the Theorem~\ref{noloss}, passing to the limit as $j\to\infty$.
\end{proof}

A consequence of this discussion is that all the gradient shrinking
Ricci solitons obtained by rescaling around the point $p\in M$ must have
the common value $\Theta(p)$ of their $\sW$--entropy $\sW(\ginf)$. In particular, if any blow--up sequence at $p$ yields an \emph{entropy--unique} limit shrinking Ricci soliton (in the sense of Definition~\ref{defunique}), then this blow--up is unique in the sense that any other blow--up at $p$ converges to the same limit. Thus we have a uniqueness result of the asymptotic ``shape'' of the singularity at $p\in M$ in this case. More generally, the uniqueness of \emph{compact} blow--up limits was obtained in \cite{SW10} (see also \cite{Ache12} for a slightly different version, using the language of $\tau$--flows).

Note that the values of the $\sW$--entropy of several gradient shrinking Ricci solitons have been computed in~\cite{CHI}.

Since all the arguments also work for every fixed function $\fp(\cdot,t)$
instead of the family of minimizers $\fph(t)$, actually 
$\lim_{t\to T}\tfpt(t)=\Theta(p)$ and all the
previous analysis can be similarly repeated.

Let us finish the paper with a discussion of the case where $p\in M$ is actually 
a singular point, that is, the point $p$ has no neighborhood on which $\abs{\Rm(\cdot,t)}_{g(t)}$ stays bounded as $t\to T$. There are several more
restrictive notions of singular points that one can consider (cf.~\cite{EMT11}).
\begin{defn} 
At a Type~I singularity of the flow, we say that $p\in M$ is a {\em
  Type~I singular point} if 
there exists a sequence of times $t_i\to T$ and points $p_i\to p$ such
that
\begin{equation*}
\abs{\Rm_{g(t_i)}(p_i)}_{g(t_i)}\geq \frac{\delta}{T-t_i}
\end{equation*}
for some constant $\delta>0$. We say that a Type~I singular point $p\in M$
is {\em special} if the points $p_i$ in the above sequence can be
chosen to be all equal to the point $p$, that is,
\begin{equation*}
\abs{\Rm_{g(t_i)}(p)}_{g(t_i)}\geq \frac{\delta}{T-t_i}
\end{equation*}
for some sequence of times $t_i\to T$ and some positive constant $\delta$.
We say that a special Type~I singular point $p\in M$ is
$\RRR$--special if there exists $\delta>0$ such that
\begin{equation*}
\abs{\RRR_{g(t_i)}(p)}_{g(t_i)}\geq \frac{\delta}{T-t_i}
\end{equation*}
for some sequence of times $t_i\to T$.
\end{defn}

By Corollary~\ref{thetaprop}, when $\Theta(p)=0$, any limit normalized gradient
shrinking Ricci soliton obtained by rescaling around $p\in M$ must satisfy
$\sW(\ginf)=0$. Hence, if $\Theta(p)=0$, by Point~5 of Lemma~\ref{solprop} the
manifold $(\Minf,\ginf)$ must be the flat $\RR^n$. Hence, any limit
gradient shrinking Ricci soliton is nontrivial if and only if
$\Theta(p)<0$. Nontriviality is easily seen to be equivalent to
non--flatness.

Notice also that this discussion implies that all the functions
$\tp(t)$ are actually \emph{negative} for every $t\in[0,T)$. Indeed,
if $\tp(t_0)=0$ then $\Theta(p)=0$ and $\tp$ is constant in the
interval $[t_0,T)$ which implies, by formula~\eqref{monotonicity},
that the original (unscaled) flow is homothetically shrinking, hence
the manifold $(M,g(t))$ is only a dilation of the limit gradient
shrinking Ricci soliton which is the flat $\RR^n$, as
$\Theta(p)=0$. This is clearly in contradiction with the fact that $M$
is compact.

The following result is related to the regularity theorem of Hein and Naber~\cite{HN12} (using a local version of the entropy functional) as well as the one of Enders--M\"{u}ller--Topping~\cite{EMT11} (using the reduced volume functional).

\begin{thm}[Points with limit density $\Theta(p)=0$ are regular points]\label{ppp1}
If $\Theta(p)=0$ then $p\in M$ cannot be a Type~I singular point of
the flow.
\end{thm}
\begin{proof}
First, we show that for every sequence $p_i\to p$ and $t_i\to T$ we have $\theta(p_i,t_i)\to 0=\Theta(p)$. Recall that $p$ is a continuity point of $\Theta:M\to\RR$ by Lemma~\ref{thcon}. Now, suppose that there exist some sequence $(p_i,t_i)$ and $\alpha>0$ such that $\theta(p_i,t_i)\to-\alpha$. For every $j\in\NN$ there exists $i_0$ such that $t_i\geq t_j$ for every $i>i_0$, hence
$\theta(p_i,t_i)\geq\theta(p_i,t_j)$. Sending $i\to\infty$ we then get
$-\alpha\geq\theta(p,t_j)$. This is clearly a contradiction, as
sending now $j\to\infty$, we have $\theta(p,t_j)\to\Theta(p)=0$.

Assume now that $p\in M$ is a Type~I singular point and $p_i\to p$, $t_i\to T$
are chosen such that for some constant $\delta>0$ there holds
$\abs{\Rm(p_i,t_i)}\geq\frac{\delta}{T-t_i}$. We consider the families
of rescaled pointed manifolds $(M,\wt{g}(s_i),p_i)$ with
$\wt{g}(s_i)=\frac{g(t_i)}{T-t_i}$ and
$s_i=-\log(T-t_i)$.

Since $\Theta(p_i)\leq 0$, we have for every $\eps>0$ (by rescaling the
integrated entropy formula~\eqref{perelman3}), setting
$\wt{f}_{p_i}(s_i)=\wh{f}_{p_i,T}(t_i)$
\begin{align*}
\eps &\geq-\theta(p_i,t_i)\geq \Theta(p_i)-\theta(p_i,t_i)\\
&=2\int_{s_i}^{+\infty}\int_M\Abs{\Rc_{\wt{g}(s)}+\nabla^2\wt{f}_{p_i}(s)- 
\frac{\wt{g}(s)}{2}}^2_{\wt{g}(s)}\frac{e^{-\wt{f}_{p_i}(s)}}{(4\pi)^{n/2}}
\,dV_{\wt{g}(s)}\,ds.
\end{align*}
Hence, by the uniform curvature estimates of Section~\ref{secfpT}, we have
\begin{equation*}
\Abs{\frac{d\,}{ds}\int_M\Abs{\Rc_{\wt{g}(s)}+\nabla^2\wt{f}_{p_i}(s)-
\frac{\wt{g}(s)}{2}}^2_{\wt{g}(s)}\frac{e^{-\wt{f}_{p_i}(s)}}{(4\pi)^{n/2}}
\,dV_{\wt{g}(s)}}\leq C
\end{equation*}
where $C=C(M,g_0,T)$ is a positive constant independent of $s$. This yields
\begin{equation*}
\eps \geq
\frac{1}{2C}\bigg(\int_M\Abs{\Rc_{\wt{g}(s_i)}+\nabla^2\wt{f}_{p_i}(s_i)-
\frac{\wt{g}(s_i)}{2}}^2_{\wt{g}(s_i)}\frac{e^{-\wt{f}_{p_i}(s_i)}}{(4\pi)^{n/2}}\,dV_{\wt{g}(s_i)}\bigg)^2.
\end{equation*}
If we argue like we did in the proof of Theorem~\ref{mainthm} at
the beginning of this section, we can extract from the sequence of
pointed manifolds $(M,\wt{g}(s_i),p_i)$ and functions
$\wt{f}_{p_i}(s_i)$ a locally smoothly converging subsequence to some
limit manifold $(\Minf,\ginf,\pinf)$ and $\finf:\Minf\to\RR$. By lower
semicontinuity of the integral in the last estimate above, we conclude
that
\begin{equation*}
\eps\geq\frac{1}{2C}\bigg(\int_{\Minf}\Abs{\Rc_{\ginf}+\nabla^2\finf-\frac{\ginf}{2}}^2_{\ginf}
\frac{e^{-\finf}}{(4\pi)^{n/2}}\,dV_{\ginf}\bigg)^2,
\end{equation*}
for every $\eps>0$, hence $(\Minf,\ginf,\finf)$ is a gradient shrinking Ricci soliton.

Finally, by Theorem~\ref{noloss}, we have
\begin{equation*}
\sW(\ginf,\finf)=\lim_{i\to\infty}\theta(p_i,t_i)=0
\end{equation*}
and by Lemma~\ref{convnorm} the soliton is normalized. Then, we conclude by Point~5 of Lemma~\ref{solprop}, that 
the soliton $(\Minf,\ginf,\finf)$ is the flat $\RR^n$.

Since, by hypothesis, at the points $p_i$ the Riemann tensor
$\Rm_{\wt{g}(s_i)}$ of $(M,\wt{g}(s_i))$ satisfies
$\abs{\Rm_{\wt{g}(s_i)}(p_i)}_{\wt{g}(s_i)}\geq \delta>0$ for every
$i\in\NN$, it follows that the Riemann tensor of the limit manifold $(\Minf,\ginf)$ 
is not zero at the point $\pinf\in
\Minf$. Hence, we have a contradiction and $p\in M$ cannot be a singular
point of the flow.
\end{proof}

An easy corollary is that every Type~I singular point has to be an
$\RRR$--special Type~I singular point (see also Enders--M\"{u}ller--Topping~\cite{EMT11}, Section 3).
\begin{cor}\label{ppp2} 
Every Type~I singular point $p\in M$ is an $\RRR$--special
Type~I singular point and there is a central blow--up converging to a
non--flat gradient shrinking Ricci soliton, hence $\Theta(p)<0$.
\end{cor}
\begin{proof}
Assume that  $p$ is not an $\RRR$--special Type~I singular point, then
\begin{equation*}
\limsup_{t\to T}(T-t)\RRR(p,t)=0.
\end{equation*}
Hence, any central blow--up limit $(\Minf,\ginf,\finf)$ will satisfy
$\RRR_{\ginf}(p)=0$. Being a gradient shrinking Ricci soliton, by Point~1 of 
Lemma~\ref{solprop}, it must be the flat $\RR^n$. But then
$\Theta(p)=0$ and $p$ cannot be a Type~I singular point of the flow,
a contradiction.
\end{proof}

Using the gap result of Yokota (Point~\ref{sol4} in Lemma~\ref{solprop}), we also obtain the following result.
\begin{cor} 
If $p\in M$ satisfies $\Theta(p)=0$ then there exists a neighborhood
$U\subset M$ of the point $p$ such that $\Theta$ is identically
zero in $U$. As a consequence, the complement of the set of the
Type~I singular points is open.
\end{cor}
\begin{proof}
Since the function $\Theta$ is lower semicontinuous and cannot attain
values between $-\eps_n$ and zero, it must be constant, hence zero, in
a neighborhood of the point $p\in M$.
\end{proof}

So far, we considered the following notions of singular points:
\begin{itemize}
\item $\Sigma_{\II}$ is the set of Type~I singular points.
\item $\Sigma_s$ is the set of special Type~I singular points.
\item $\Sigma_{\RRR}$ is the set of $\RRR$--special Type~I singular points.
\item $\Sigma_\Theta$ is the set of points $p\in M$ where $\Theta(p)<0$.
\item $\Sigma_{\eps_n}$  is the set of points $p\in M$ where $\Theta(p)\leq-\eps_n$.
\end{itemize}
Trivial inclusions are
$\Sigma_{\RRR}\subset\Sigma_s\subset\Sigma_{\II}$ and
$\Sigma_{\eps_n}\subset\Sigma_\Theta$. The previous corollary says that the
last two sets actually coincide, Theorem~\ref{ppp1} shows that
$\Sigma_{\II}=\Sigma_\Theta$ and Corollary~\ref{ppp2} proves that
$\Sigma_{\RRR}=\Sigma_{\II}$, hence \emph{all these sets coincide} and they
are all contained in the set $\Sigma$ of points $p\in M$ for which
there exists a sequence of points $p_i\to p$ and times $t_i\to T$ such
that $\abs{\Rm_{g(t_i)}(p_i,t_i)}\to+\infty$, that is, the most general set of
singular points.

Theorem~3.3 in~\cite{EMT11} shows that also $\Sigma$ coincides with all these sets.
\begin{prop}[see {\cite[Thm.~3.3]{EMT11}}]
We have $\Sigma_{\RRR}=\Sigma_s=\Sigma_{\II}=\Sigma_{\eps_n}=\Sigma_\Theta=\Sigma$, which is a closed subset of $M$.
\end{prop}

Combining this result with Corollary~\ref{ppp2}, Theorem~\ref{nontriv} follows.

\makeatletter
\def\@listi{%
  \itemsep=0pt
  \parsep=1pt
  \topsep=1pt}
\makeatother
{\fontsize{10}{11}\selectfont

}
\vspace{10mm}

Carlo Mantegazza\\
{\sc Scuola Normale Superiore di Pisa, 56126 Pisa, Italy}\\

Reto M\"uller\\
{\sc Imperial College London, London SW7 2AZ, United Kingdom}

\begin{thebibliography}{99}

\bibitem{Ache12}
A.~Ache.
On the uniqueness of asymptotic limits of the Ricci flow.
Preprint 2012, {\em ArXiv:1211.3387v1}.

\bibitem{CHI}
H.--D.~Cao, R.~Hamilton, and T.~Ilmanen.
Gaussian densities and the stability for some Ricci solitons.
Preprint 2004, {\em ArXiv:math/0404165v1}.

\bibitem{CZ10}
H.--D.~Cao and D.~Zhou.
On complete gradient shrinking Ricci solitons.
{\em J. Diff. Geom.}, 85:175--185, 2010.

\bibitem{CZ11}
X.~Cao and Q.~Zhang.
The conjugate heat equation and ancient solutions of the Ricci flow.
{\em Adv. Math.}, 228(5): 2891--2919, 2011.

\bibitem{CN08}
J.~Carrillo and L.~Ni.
Sharp logarithmic Sobolev inequalities on gradient solitons and applications.
{\em Comm. Anal. Geom.}, 17: 721--753, 2009.

\bibitem{RFTA3}
B.~Chow et al.
The Ricci flow: techniques and applications. Part III. Geometric--analytic aspects.
{\em Mathematical Surveys and Monographs}, 163, AMS, 2010.

\bibitem{CLN06}
B.~Chow, P.~Lu and L.~Ni.
Hamilton's Ricci flow.
{\em Graduate Studies in Mathematics}, AMS, 2006.

\bibitem{EMT11}
J.~Enders, R.~M\"{u}ller and P.~Topping.
On Type~I singularities in Ricci flow.
{\em Comm. Anal. Geom.}, 19: 905--922, 2011.

\bibitem{Ham82}
R.~S.~Hamilton.
Three--manifolds with positive Ricci curvature.
{\em J. Diff. Geom.}, 17:255--306, 1982.

\bibitem{Ham95}
R.~S.~Hamilton.
A compactness property for solutions of the Ricci flow.
{\em Amer. J. Math.}, 117:545--572, 1995.

\bibitem{HM11}
R.~Haslhofer and R.~M\"{u}ller.
A compactness theorem for complete Ricci shrinkers.
{\em Geom. Funct. Anal.}, 21: 1091--1116, 2011.

\bibitem{HN12}
H.--J.~Hein and A.~Naber.
New logarithmic Sobolev inequalities and an $\eps$-regularity theorem for the Ricci flow.
Preprint 2012, {\em ArXiv:1205.0380v1}.

\bibitem{Hui90}
G.~Huisken.
Asymptotic behavior for singularities of the mean curvature flow.
{\em J. Diff. Geom.}, 31:285--299, 1990.

\bibitem{LS10}
N.~Le and N.~Sesum.
Remarks on curvature behavior at the first singular time of the Ricci flow.
Preprint 2010, {\em ArXiv:1005.1220v2}.

\bibitem{MT07}
J.~Morgan and G.~Tian.
Ricci flow and the Poincar\'{e} conjecture.
Clay Math. Monographs 3, AMS, Providence, 2007.

\bibitem{Mul06}
R.~M\"{u}ller.
Differential Harnack inequalities and the Ricci flow.
{\em EMS Series of Lectures in Mathematics}, 2006.

\bibitem{Mun09}
O.~Munteanu.
The volume growth of complete gradient shrinking Ricci solitons.
Preprint 2009, {\em ArXiv:0904.0798v2}.

\bibitem{Nab10}
A.~Naber.
Noncompact shrinking 4--solitons with nonnegative curvature.
{\em J. Reine Angew. Math.}, 645: 125--153, 2010.

\bibitem{Ni06}
L.~Ni.
A note on Perelman's LYH--type inequality.
{\em Comm. Anal. Geom.} 14:883--905, 2006.

\bibitem{Per02}
G.~Perelman.
The entropy formula for the Ricci flow and its geometric applications.
Preprint 2002, {\em arXiv:math/0211159v1}.

\bibitem{Per03}
G.~Perelman.
Ricci flow with surgery on three--manifolds .
Preprint 2003, {\em arXiv:math/0303109v1}.

\bibitem{PRS11}
S.~Pigola, M.~Rimoldi and A.~G.~Setti.
Remarks on non--compact gradient Ricci solitons.
{\em Math. Z.}, 268(3--4): 777--790, 2011.

\bibitem{Shi89}
W.--X.~Shi.
Deforming the metric on complete Riemannian manifolds.
{\em J. Diff. Geom.}, 30: 223--301, 1989.

\bibitem{Sto94}
A.~Stone.
A density function and the structure of singularities of the mean curvature flow.
{\em Calc. Var. PDE}, 2: 443--480, 1994.

\bibitem{SW10}
S.~Sun and Y.~Wang.
On the K\"{a}hler--Ricci flow near a K\"{a}hler--Einstein metric.
Preprint 2010, {\em ArXiv:1004.2018v2}.

\bibitem{Yok09}
T.~Yokota.
Perelman's reduced volume and a gap theorem for the Ricci flow.
{\em Comm. Anal. Geom.}, 17:227--263, 2009.

\bibitem{Zha09}
Z.--H.~Zhang.
On the completeness of gradient Ricci solitons.
{\em Proc. Amer. Math. Soc.}, 137(8): 2755--2759, 2009.

\end{thebibliography}
\end{document}